\documentclass[11pt]{amsart}

\usepackage{epsfig,amsmath,amsfonts,latexsym,mathrsfs,amssymb,flafter}
\usepackage{color}
\usepackage{overpic}
 \usepackage{palatino}
 \usepackage{hyperref}
  \usepackage{graphics}
  \usepackage{mathtools}
  \usepackage{adjustbox}
 \usepackage[nocompress]{cite}
 \usepackage{tikz,tikz-cd}
\usetikzlibrary{calc,decorations.pathmorphing,shapes,cd}

\newcounter{sarrow}

\newtheorem{theorem}{Theorem}[section]
\newtheorem{proposition}[theorem]{Proposition}
\newtheorem{lemma}[theorem]{Lemma}
\newtheorem{corollary}[theorem]{Corollary}

\theoremstyle{definition}

\newtheorem{remark}[theorem]{Remark}

\theoremstyle{remark}

\DeclareFontFamily{U}{mathx}{\hyphenchar\font45}
\DeclareFontShape{U}{mathx}{m}{n}{
      <5> <6> <7> <8> <9> <10>
      <10.95> <12> <14.4> <17.28> <20.74> <24.88>
      mathx10
      }{}
\DeclareSymbolFont{mathx}{U}{mathx}{m}{n}
\DeclareFontSubstitution{U}{mathx}{m}{n}
\DeclareMathAccent{\widecheck}{0}{mathx}{"71}

\theoremstyle{plain}

\begin{document}

\title{Unknotted Curves on Seifert Surfaces}

\author{Subhankar Dey}

\author{Veronica King}

\author{Colby T. Shaw}

\author{B\"{u}lent Tosun}

\author{Bruce Trace}

\address{Department of Mathematics\\ University of Alabama\\Tuscaloosa\\AL}

\email{sdey4@ua.edu}

\address{Department of Mathematics\\ University of Texas Austin\\Austin\\TX}

\email{viking0905@gmail.edu}

\address{School of Mathematics\\ Georgia Institute of Technology\\Atlanta\\GA}

\email{cshaw44@gatech.edu}

\address{Department of Mathematics\\ University of Alabama\\Tuscaloosa\\AL}

\email{btosun@ua.edu}

\address{Department of Mathematics\\ University of Alabama\\Tuscaloosa\\AL}

\email{btrace@ua.edu}


\begin{abstract} We consider homologically essential simple closed curves on Seifert surfaces of genus one knots in $S^3$, and in particular those that are unknotted or slice in $S^3$. We completely characterize all such curves for most twist knots: they are either positive or negative braid closures; moreover, we determine exactly which of those are unknotted. A surprising consequence of our work is that the figure eight knot admits infinitely many unknotted essential curves up to isotopy on its genus one Seifert surface, and those curves are enumerated by Fibonacci numbers. On the other hand, we prove that many twist knots admit homologically essential curves that cannot be positive or negative braid closures. Indeed, among those curves, we exhibit an example of a slice but not unknotted homologically essential simple closed curve. We further investigate our study of unknotted essential curves for arbitrary Whitehead doubles of non-trivial knots, and obtain that there is a precisely one unknotted essential simple closed curve in the interior of the doubles' standard genus one Seifert surface. As a consequence of all these we obtain many new examples of 3-manifolds that bound contractible 4-manifolds.   
\end{abstract}

\maketitle

\section{Introduction}

Suppose $K \subseteq S^3$ is a genus $g$ knot with Seifert Surface $\Sigma_K.$ Let $b$ be a curve in $\Sigma_K$ which is {\it homologically essential}, that is it is not separating $\Sigma_K$, and {\it a simple closed curve}, that is it has one component and does not intersect itself.  Furthermore, we will focus on those that are {\it unknotted} or {\it slice} in $S^3$, that is each bounds a disk in $S^3$ or $B^4$. In this paper we seek to progress on the following problem:

\medskip

\noindent {\it Characterize and, if possible, list all such b's for the pair $(K, \Sigma_K)$ where $K$ is a genus one knot and $\Sigma_K$ its Seifert surface.}

\medskip

Our original motivation for studying this problem comes from the intimate connection between unknotted or slice homologically essential curves on a Seifert surface of a genus one knot and $3$-manifolds that bound contractible $4$-manifolds. We defer the detailed discussion of this connection to Section~\ref{unknottocontractible}, where we also provide some historical perspective. For now, however, we will focus on getting a hold on the stated problem above for a class of genus one knots, and as we will make clear in the next few results, this problem is already remarkably interesting and fertile on its own.

\subsection{Main Results.} A well studied class of genus one knots is so called twist knot $K=K_t$ which is described by the diagram on the left of Figure~\ref{TwistKnots}. We note that with this convention $K_{-1}$ is the right-handed trefoil $T_{2,3}$ and $K_1$ is the figure eight knot $4_1$. We will consider the genus one Seifert surface $\Sigma_K$ for $K=K_t$ as depicted on the right of Figure~\ref{TwistKnots}.

\begin{figure}[h!]
\begin{center}
 \includegraphics[width=10cm]{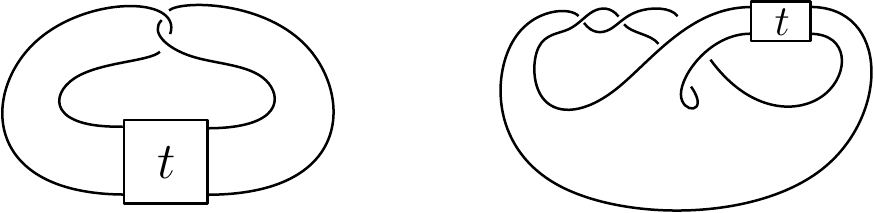}
 \caption{On the left is the twist knot $K_t$ where the box contains $t$ full right-handed twists if $t\in\mathbb{Z}_{>0}$, and $|t|$ full left-handed twists if $t\in\mathbb{Z}_{<0}$. On the right is the standard Seifert surface for $K_t$. }
  \label{TwistKnots}
\end{center}
\end{figure} 

The first main result in this paper is the following.

\begin{theorem}\label{main}
Let $t\leq 2$. Then the genus one Seifert surface $\Sigma_K$ of $K=K_t$ admits infinitely many homologically essential, unknotted curves, if and only if $t=1$, that is $K$ is the figure eight knot $4_1.$ 
\end{theorem} 

\noindent Indeed, we can be more precise and characterize all homologically essential, simple closed curves on $\Sigma_K$, from which Theorem~\ref{main} follows easily. To state this we recall an essential simple closed curve $c$ on $\Sigma_K$ can be represented (almost uniquely) by a pair of non-negative integers $(m,n)$ where $m$ is the number of times $c=(m,n)$ runs around the left band and $n$ is the number of times it runs around the right band in $\Sigma_{K}$. Moreover, since $c$ is connected, we can assume $\textrm{gcd}(m,n) = 1$. Finally, to uniquely describe $c$, we adopt the notation of {\it $\infty$ curve} and {\it loop curve} for a curve $c$, if the curve has its orientation switches one band to the other and it has the same orientation on both bands, respectively (See Figure~\ref{Setup}). 

\begin{theorem} \label{Twist} Let $K=K_t$ be a twist knot and $\Sigma_K$ its Seifert surface as in Figure~\ref{TwistKnots}. Then; 
\begin{enumerate}

\item For $K =K_{t}$, $t\leq -1$, we can characterize all homologically essential simple closed curves on $\Sigma_K$ as the closures of negative braids in Figure \ref{NTC}. In case of the right-handed trefoil  $K_{-1}=T_{2,3}$, exactly $6$ of these, see Figure \ref{unknottedcurvestref}, are unknotted in $S^3.$ For $t<-1$,  exactly $5$ of these, see Figure \ref{twistunknots}, are unknotted in $S^3.$

\item For $K=K_1=4_1$, we can characterize all homologically essential simple closed curves on $\Sigma_K$ as the closures of braids in Figure \ref{41Braids}. A curve on this surface is unknotted in $S^3$ if and only if it is (1) a trivial curve $(1,0)$ or $(0,1)$, (2) an $\infty$ curve in the form of $(F_{i+1},F_{i})$, or (3) a loop curve in the form of $(F_{i},F_{i+1})$, where $F_i$ represents the $i^{th}$ Fibonacci number, see Figure~\ref{41unknotss}.

\end{enumerate}
\end{theorem}

\begin{figure}[h] 
    \centering
    \includegraphics[scale = .85]{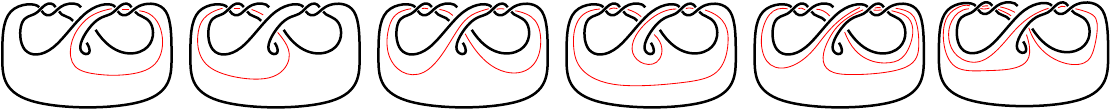}
       \caption{It can easily be shown these $6$ curves, from left to right $(0,1), (1,0), (1,1)~\infty, (1,1)~\textrm{loop}, (1,2)~\infty$ and $(2,1)~\infty$, on $\Sigma_K$ are unknotted in $S^3.$ One can easily check that the other $(1,2)$ and $(2,1)$ curves ( that is $(1,2)~\textrm{loop}$ and $(2,1)~\textrm{loop}$ curves) both yield the left-handed trefoil $T_{2,-3}$, and hence they are not unknotted in $S^3.$ }
    \label{unknottedcurvestref}
\end{figure}

\begin{figure}[h] 
    \centering
    \includegraphics[scale = .5]{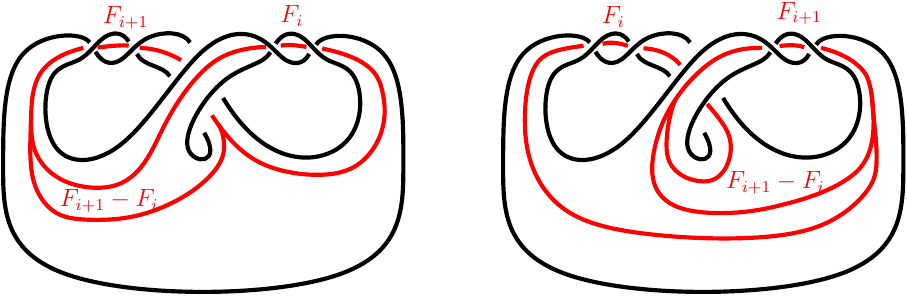}
       \caption{The two infinite families of unknotted curves for the figure eight knot in $S^3.$ The letters on parts of our curve or in certain locations stands for the number of strands that particular curve or location. For example, for the $(m,n)$ $\infty$ curve on the left we will show in Section~\ref{fep} via explicit isotopies how starting with the known unknotted $(1, 1)$ $\infty$ curve we can recursively obtain the following sequence of unknotted curves: $(1,1) \sim (3,2) \sim (8,5) \sim (21, 13) \sim (55, 34) \sim \cdots$ }
    \label{41unknotss}
\end{figure}

\noindent For twist knot $K=K_t$ with $t>1$ the situation is more complicated. Under further hypothesis on the parameters $m,n$ we can obtain results similar to those in Theorem~\ref{Twist}, and these will be enough to extend the theorem entirely to the case of $K=K_2$, so called Stevedore's knot $6_1$ (here we use the Rolfsen's knot tabulation notation). More precisely we have;

\begin{theorem}\label{Twist2} Let $K=K_t$ be a twist knot and $\Sigma_K$ its Seifert surface as in Figure~\ref{TwistKnots}. Then; 

\begin{enumerate}
\item When $t>1$ and $m<n$, we can characterize all homologically essential simple closed curves on $\Sigma_K$ as the closures of positive braids in Figure~\ref{TTC}(a)(b). Exactly $5$ of these, see Figure \ref{twistunknots}, are unknotted in $S^3.$  

\item  When $t>1$ and $m>n$.
\begin{enumerate}
\item If $m-tn>0$, then we can characterize all homologically essential simple closed curves on $\Sigma_K$ as the closures of negative braids in Figure~\ref{TTCase3NN} and ~\ref{TTCase4NN}. Exactly $5$ of these, see Figure \ref{twistunknots}, are unknotted in $S^3.$ 

\item If $m-n<n$ and the curve is $\infty$ curve, then we can characterize all homologically essential simple closed curves on $\Sigma_K$ as the closures of positive braids Figure \ref{Case3sporadic}. Exactly $5$ of these, see Figure \ref{twistunknots}, are unknotted in $S^3.$ 

\end{enumerate}

\item For $K =K_{2}=6_1$, we can characterize all homologically essential simple closed curves on $\Sigma_K$ as the closures of positive or negative braids. Exactly $5$ of these, see Figure \ref{twistunknots}, are unknotted in $S^3.$ 

\end{enumerate}
\end{theorem}   

What Theorem~\ref{Twist2} cannot cover is the case $t>2$, $m>n$ and $m-tn<0$ or when $m-n<n$ and the curve is a loop curve. Indeed in this range {\it not} every homologically essential curve is a positive or negative braid closure. For example, when $(m,n)=(5,2)$ and $t=3$ one obtains that the corresponding essential $\infty$ curve, as a smooth knot in $S^3$, is the knot $5_2$, and for $(m,n)=(7,3)$ and $t=3$, the corresponding knot is $10_{132}$ both of which are known to be not positive braid closures--coincidentally, these knots are not unknotted or slice. Moreover we can explicitly demonstrate, see below, that if one removes the assumption of ``$\infty$'' from part 2(b) in Theorem~\ref{Twist2}, then the conclusion claimed there fails for certain loop curves when $t>2$. A natural question is then whether for knot $K =K_{t}$ with $t>2$, $m>n$ and $m-tn<0$ or $m-n<n$ loop curve, there exists unknotted or slice curves on $\Sigma_K$ other than those listed in Figure \ref{twistunknots}? A follow up question will be whether there exists slice but not unknotted curves on $\Sigma_K$ for some $K=K_t$? We can answer the latter question in affirmative as follows:

\begin{theorem}\label{Twist3}
 Let $K=K_t$ be a twist knot with $t>2$ and $\Sigma_K$ its Seifert surface as in Figure~\ref{TwistKnots} and consider the loop curve $(m,n)$ with $m=3, n=2$ on $\Sigma_K$. Then this curve, as a smooth knot in $S^3$, is the pretzel knot $P(2t-5, -3, 2)$. This knot is never unknotted but it is slice (exactly) when $t=4$, in which case this pretzel knot is also known as the curious knot $8_{20}$. 
\end{theorem}

\begin{remark}
We note that the choices of $m,n$ values made in Theorem~\ref{Twist3} are somewhat special in that they yielded an infinite family of pretzel knots, and that it includes a slice but not unknotted curve. Indeed, by using Rudolph's work in \cite{rudolph}, we can show (see Proposition~\ref{Rtwist}) that the loop curve $(m,n)$ with $m-n=1, n>2$ and $t>4$ on $\Sigma_K$, as a smooth knot in $S^3$, is never slice. The calculation gets quickly complicated once $m-n>1$, and it stays an open problem if in this range one can find other slice but not unknotted curves.    
\end{remark}

\begin{figure}[h] 
    \centering
    \includegraphics[scale = .85]{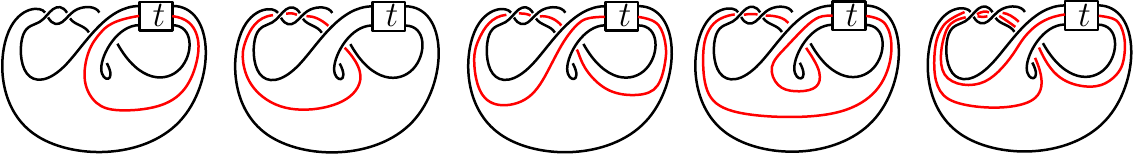}
       \caption{These $5$ curves, from left to right $(0,1), (1,0), (1,1)~\infty, (1,1)~\textrm{loop}$ and $(2,1)~\infty$, on $\Sigma_{K}$ where $K=K_t$, $t\neq 1~\textrm{or}~-1$, are unknotted curves in $S^3.$} %One can check the other $(2,1)$ curve (i.e. $(2,1)~\textrm{loop}$) will be the left-handed trefoil, so not unknotted, in $S^3.$ }
    \label{twistunknots}
\end{figure}

\noindent We can further generalize our study of unknotted essential curves on minimal genus Seifert surface of genus one knots for the Whitehead doubles of non-trivial knots. We first introduce some notation. Let $P$ be the twist knot $K_t$ embedded (where $t=0$ is allowed) in a solid torus $V\subset S^3$, and $K$ denote an arbitrary knot in $S^3$, we identify a tubular neighborhood of $K$ with $V$ in such a way that the longitude of $V$ is identified with the longitude of $K$ coming from a Seifert surface. The image of $P$ under this identification is a knot, $D^{\pm}(K,t)$, called the positive/negative $t$--twisted Whitehead double of $K$. In this situation the knot $P$ is called the pattern for $D^{\pm}(K,t)$ and $K$ is referred to as the companion. Figure~\ref{doubleSeifert} depicts the positive $-3$--twisted Whitehead double of the left-handed trefoil, $D^+(T_{2,-3}, -3)$. If one takes $K$ to be the unknot, then $D^{+}(K,t)$ is nothing but the twist knot $K_t$.

\begin{figure}[h!]
\begin{center}
 \includegraphics[width=10cm]{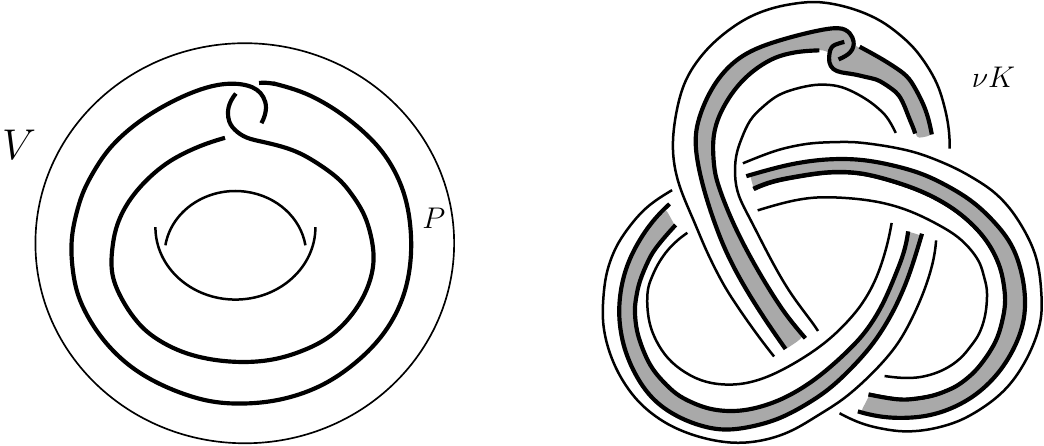}
 \caption{On the right is the solid torus $V\subset S^3$ and the pattern twist knot $P$ (which in this case $t=0$). On the left is the positive $-3$--twisted Whitehead double of the left-handed trefoil, and its standard genus one Seifert surface.}
  \label{doubleSeifert}
\end{center}
\end{figure}

\begin{theorem}\label{Whitehead}
Let $K$ denote a non-trivial knot in $S^3$. Suppose that $\Sigma_K$ is a standard genus one Seifert surface for the Whitehead double of $K$. Then there is precisely one unknotted homologically essential, simple closed curves in the interior of $\Sigma_K$.  
\end{theorem}

\subsection{From unknotted curves to contractible 4-manifolds.}\label{unknottocontractible} The problem of finding unknotted homologically essential curves on a Seifert surface of a genus one knot is interesting on its own, but it is also useful for studying some essential problems in low dimensional topology. We expand on one of these problems a little more. An important and still open question in low dimensional topology asks: {\it which closed oriented homology 3-sphere \footnote{A homology 3-sphere/4-ball is a 3-/4- manifold having the integral homology groups of $S^3/B^4$.} bounds a homology 4-ball or contractible 4-manifold} (see \cite[Problem~$3.20$]{Kirby:problemlist}). This problem can be traced back to the famous Whitney embedding theorem and other important subsequent results due to Hirsch, Wall and Rokhlin \cite{Hirsch, Wall:embedding, Rohlin:3manembedding} in the 1950s. Since then the research towards understanding this problem has stayed active. It has been shown that many infinite families of homology spheres do bound contractible 4-manifolds \cite{CH, Fickle, Stern, Zeeman} and at the same time many powerful techniques and invariants, mainly coming from Floer and gauge theories \cite{Manolescu:T, FintushelStern84, Rohlin:3manembedding} have been used to obtain constraints.  

In our case, using our main results, we will be able to list some more homology spheres that bound contractible 4-manifolds. This is because of the following theorem of Fickle \cite[Theorem~3.1]{Fickle}.

\begin{theorem}[Fickle]\label{fickle}
Let $K$ be a knot in $S^3$ which has a genus one Seifert surface $F$ with a primitive element $[b]\in H_1(F)$ such that the curve $b$ is unknotted in $S^3.$ If $b$ has self-linking $s$, then the homology sphere obtained by $\frac{1}{(s\pm 1)}$ Dehn surgery on $K$ bounds a contractible \footnote{Indeed, this contractible manifold is a {\it Mazur-type} manifold, namely it is a contractible $4$-manifold that has a single handle of each index $0, 1$ and $2$ where the $2$-handle is attached along a knot that links the $1$-handle algebraically once. This condition yields a trivial fundamental group.} 4-manifold. 
\end{theorem}
 
\noindent This result in \cite[Theorem~$1$]{ET} was generalized to genus one knots in the boundary of an acyclic $4$--manifold $W$, and where the assumption on the curve $b$ is relaxed so that $b$ is slice in $W.$ This will be useful for applying to the slice but not unknotted curve/knot found in Theorem~\ref{Twist3}.

\smallskip

\noindent The natural task is to determine self-linking number $s$, with respect to the framing induced by the Seifert surface, for the unknotted curves found in Theorem~\ref{Twist} and \ref{Whitehead}. For this we use the Seifert matrix given by $ S = \begin{psmallmatrix}-1 & -1\\ 0 & t\end{psmallmatrix}$ where we use two obvious cycles--both oriented counterclockwise--in $\Sigma_K$. Recall that, if  $c=(m,n)$ is a loop curve then $m$ and $n$ strands are endowed with the same orientation and hence the same signs. On the other hand for $\infty$ curve they will have opposite orientation and hence the opposite signs. Therefore, given $t$, the self-linking number of $c=(m,n)$ loop curve is $s=-m^2-mn+n^2t$, and the self-linking number of $(m,n)$ $\infty$ curve is $s=-m^2+mn+n^2t$.  A quick calculation shows that the six unknotted curves in Figure~\ref{unknottedcurvestref} for $K_{-1}=T_{2,3}$ share self-linking numbers $s=-1, -3$. As we will see during the proof of Theorem~\ref{Twist} the infinitely many unknotted curves for the figure eight knot $K_1=4_1$ reduce (that are isotopic) to unknotted curves with $s=-1$ or $s=1$. The five unknotted curves in Figure~\ref{twistunknots} for $K_t$, $t<-1$ or $t>1$, share self-linking numbers $s=-1, t$ and $t-2$ (see \cite{CD} and references therein for some relevant work). Finally, Theorem~\ref{Twist3} finds a slice but not unknotted curve which is the curve $(3,2)$ with $t=4$. One can calculate from the formula above that this curve has self-linking number $s=1$. Finally, the unique unknotted curve from Theorem~\ref{Whitehead} has self linking $s=-1$. Thus, as an obvious consequence of these calculations and Theorem \ref{fickle} and its generalization in \cite{ET} we obtain:

\begin{corollary} Let $K$ be any non-trivial knot. Then, the homology spheres obtained by
\begin{enumerate}
\item $-\frac{1}{2}$ Dehn surgery on $D^{+}(K,t)$
\item $\pm\frac{1}{2}$ Dehn surgery on $K_1=4_1$
\item $-\frac{1}{2}$ and $-\frac{1}{4}$ Dehn surgeries on $K_{-1}=T_{2,3}$
\item $-\frac{1}{2}$ and $\frac{1}{t\pm1}$ and $\frac{1}{(t-2)\pm1}$ Dehn surgeries on $K_t$, $t\neq \pm 1$ 
\item  $~~\frac{1}{2}$ Dehn surgery on $K_4$ 
\end{enumerate}
bound contractible $4$-manifolds.
\end{corollary}
 
\begin{remark}
The 3-manifolds in part (3) are Brieskorn spheres $\Sigma(2,3,13)$ and $\Sigma(2,3,25)$; they were identified by Casson-Harer and Fickle that they bound contractible 4-manifolds. Also, it was known already that the result of $\frac{1}{2}$ Dehn surgery on the figure eight knot bounds a contractible 4-manifold (see \cite[Theorem~$18$]{tosun:survey}) from this we obtain the result in part (2) as the figure eight knot is an amphichiral knot.  
\end{remark}

\begin{remark}
It is known that the result of $\frac{1}{n}$ Dehn surgery on a slice knot $K\subset S^3$ bounds a contractible 4-manifold. To see this, note that at the 4-manifold level with this surgery operation what we are doing is to remove a neighborhood of the slice disk from $B^4$ (the boundary at this stage is zero surgery on $K$) and then attach a 2-handle to a meridian of $K$ with framing $-n$. Now, simple algebraic topology arguments shows that this resulting 4-manifold is contractible. 

It is a well known result that \cite{CG}; a nontrivial twist knot $K=K_t$ is slice if and only if $K=K_2$ (Stevedore's knot $6_1$). So, by arguments above we already know that result of $\frac{1}{n}$ surgery on $K_2$ bounds contractible 4-manifold for any integer $n$. But interestingly we do not recover this by using Theorem~\ref{Twist2}.    
\end{remark}

The paper is organized as follows. In Section~\ref{prel} we set some basic notations and conventions that will be used throughout the paper. Section~\ref{proofs} contains the proofs of Theorem~\ref{Twist}, ~\ref{Twist2} and ~\ref{Twist3}. Our main goal will be to organize, case by case, essential simple closed curves on genus one Seifert surface $\Sigma_K$, through sometimes lengthy isotopies, into explicit positive or negative braid closures. Once this is achieved we use a result due to Cromwell that says the Seifert algorithm applied to the closure of a positive/negative braid closure gives a minimal genus surface. This together with some straightforward calculations will help us to determine the unknotted curves exactly. But sometimes it will not be obvious or even possible to reduce an essential  simple closed curve to a positive or negative closure (see Section~\ref{fep}, ~\ref{ptp} and ~\ref{ptp2}). Further analyzing these cases will yield interesting phenomenon listed in Theorem ~\ref{Twist2} and ~\ref{Twist3}. Section~\ref{white} contains the proof of Theorem~\ref{Whitehead}.

\subsection*{Acknowledgments} We thank Audrick Pyronneau and Nicolas Fontova for helpful conversations. The first, second and third authors were supported in part by a grant from NSF (DMS-2105525). The fourth author was supported in part by grants from NSF (CAREER DMS-2144363 and DMS-2105525) and the Simons Foundation (636841, BT).

\section{Preliminaries}\label{prel} In this section, we set some notation and make preparations for the proofs in the next three sections. In Figure~\ref{SC} we record some basic isotopies/conventions that will be repeatedly used during proofs. Most of these are evident but for the reader's convenience we explain how the move in part (f) works in Figure~\ref{QuestionableMove}. We remind the reader that letters on parts of our curve, as in part $(e)$ of the figure, or in certain location is to denote the number of strands that particular curve has.
 
\begin{figure}[h!]
\begin{center}
 \includegraphics[width=12cm]{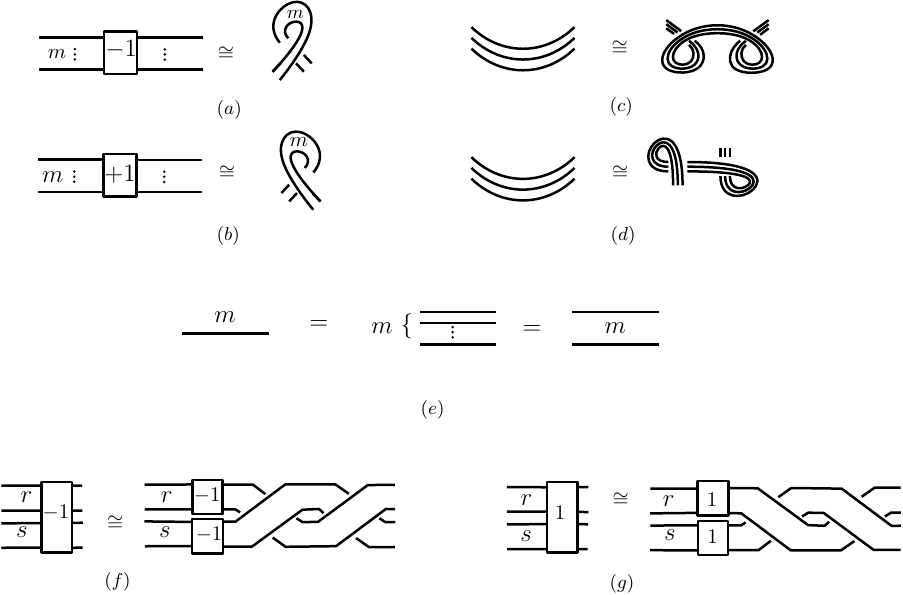}
 \caption{Various isotopies.}
  \label{SC}
\end{center}
\end{figure}

\begin{figure}[h!]
\begin{center}
 \includegraphics[width=11cm]{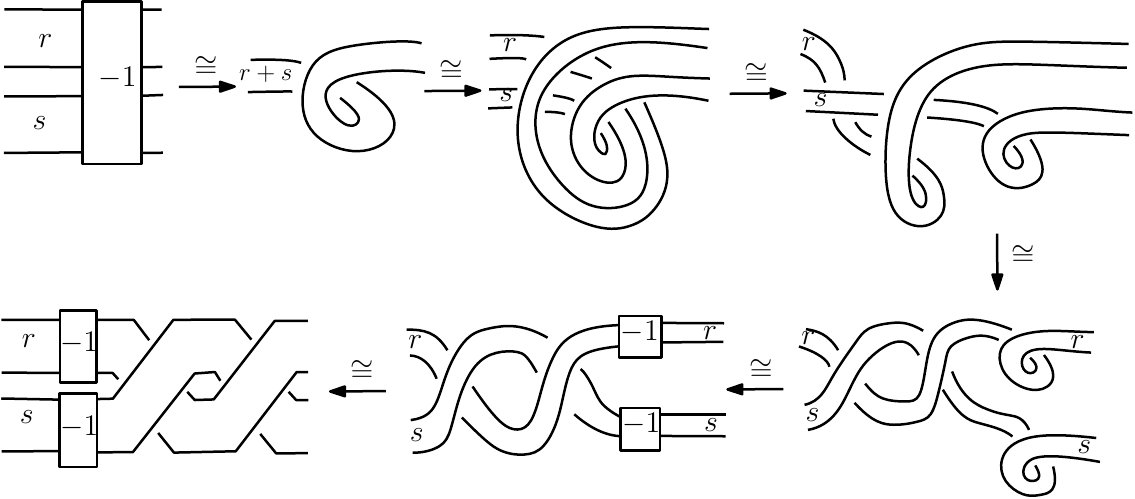}
 \caption{Diagrammatic proof of move in Figure~\ref{SC}(f).}
  \label{QuestionableMove}
\end{center}
\end{figure}

Recall also an essential, simple closed curve on $\Sigma_K$ can be represented by a pair of non-negative integers $(m,n)$ where $m$ is the number of times it runs around the left band and $n$ is the number of times it runs around the right band in $\Sigma_{K}$, and since we are dealing with connected curves we must have that $m,n$ are relatively prime. 

We have two cases: $m>n$ or $n>m$. For an $(m,n)$ curve with $m>n$, after the $m$ strands pass under the $n$ strands on the Seifert surface, it can be split into two sets of strands. For this case, assume that the top set is made of $n$ strands. They must connect to the $n$ strands going over the right band, leaving the other set to be made of $m-n$ strands. Now, we can split the other side of the set of $m$ strands into two sections. The $m-n$ strands on the right can only go to the bottom of these two sections, because otherwise the curve would have to intersect itself on the surface. This curve is notated an $(m,n)\ \infty$ curve. See Figure~\ref{Setup}(a). The other possibility for an $(m,n)$ curve with $m>n$, has $n$ strands in the bottom set instead, which loop around to connect with the $n$ strands going over the right band. This leaves the other to have $m-n$ strands. We can split the other side of the set of $m$ strands into two sections. The $m-n$ strands on the right can only go to the top of these two sections, because again otherwise the curve would have to intersect itself on the surface. The remaining subsection must be made of $n$ strands and connect to the $n$ strands going over the right band. This curve is notated as an $(m,n)$ loop curve. See Figure~\ref{Setup}(b). The case of $(m,n)$ curve with $n>m$ is similar. See Figure~\ref{Setup}(c)$\&$(d).

\begin{figure}[h!]
\begin{center}
 \includegraphics[width=11cm]{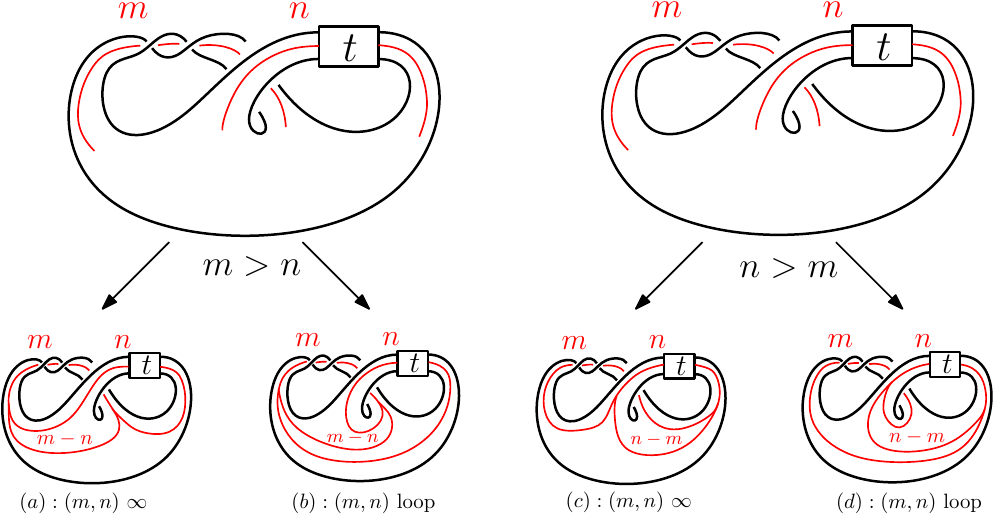}
 \caption{Possibilities for an essential, simple closed curve $(m,n)$ on $\Sigma_K$.}
  \label{Setup}
\end{center}
\end{figure}

\section{Twist Knots}\label{proofs} In this section we provide the proofs of Theorem~\ref{Twist}, \ref{Twist2} and ~\ref{Twist3}. We do this in four parts. Section~\ref{ntp} and \ref{fep} contains all technical details of Theorem~\ref{Twist}, Section~\ref{ptp} contains details of Theorem~\ref{Twist2} and Section~\ref{ptp2} contains Theorem~\ref{Twist3} .

\subsection{Twist knot with $t<0$}\label{ntp} In this section we consider twist knot $K=K_t$, $t\leq -1$. This in particular includes the right-handed trefoil $K_{-1}$. 

\begin{proposition} \label{charnegtwist}
All essential, simple closed curves on $\Sigma_K$ can be characterized as the closure of one of the negative braids in Figure~\ref{NTC}.
\end{proposition}

\begin{figure}[h!]
\begin{center}
 \includegraphics[width=11cm]{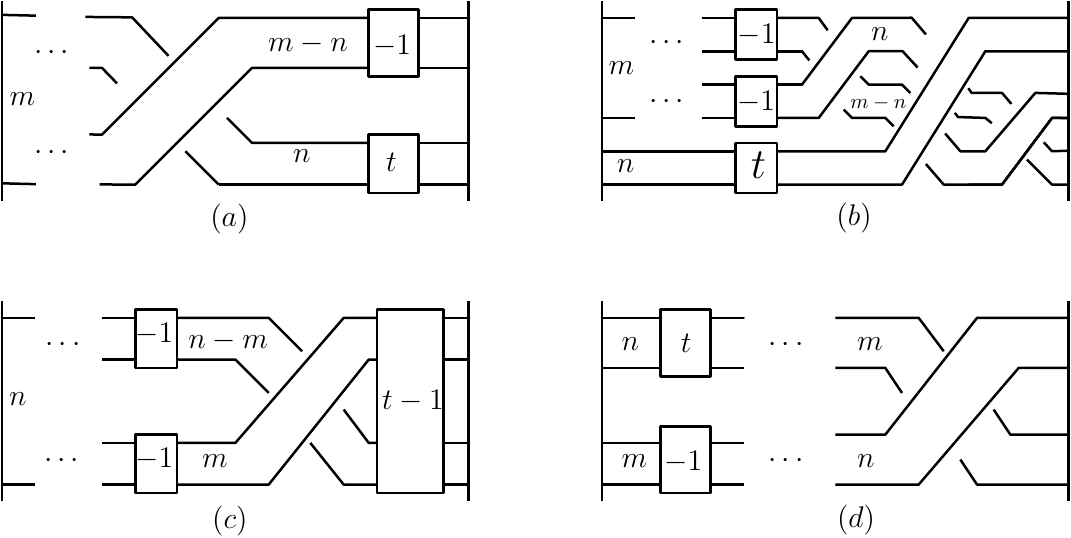}
 \caption{}
  \label{NTC}
\end{center}
\end{figure}

\begin{proof} It suffices to show all possible curves for an arbitrary $m$ and $n$ such that $gcd(m, n) = 1$ are
the closures of either braid in Figure~\ref{NTC}. As mentioned earlier we will deal with cases where both $m, n \geq1$ since cases involving $0$ are trivial. There are four cases to consider. The arguments for each of these will be quite similar, and so we will explain the first case in detail and refer to to the rather self-explanatory drawings/figures for the remaining cases.

\medskip

{\it Case 1:  $(m,n)\ \infty$ curve with $m>n>0$.} This case is explained in Figure~\ref{NTCase1}. The picture on top left is the $(m,n)$ curve we are interested. The next picture to its right is the $(m,n)$ curve where we ignore the surface it sits on and use the convention from Figure~\ref{SC}(e). The next picture is an isotopy where we push the split between $n$ strands and $m-n$ strands along the dotted blue arc. The next three pictures are obtained by applying simple isotopies coming from Figure~\ref{SC}. For example, the passage from the bottom right picture to one to its left is via  Figure~\ref{SC}(c). Finally, the picture on the bottom left, one can easily see that, is the closure of the negative braid depicted in Figure~\ref{NTC}(a).    

\begin{figure}[h!]
\begin{center}
 \includegraphics[width=13cm]{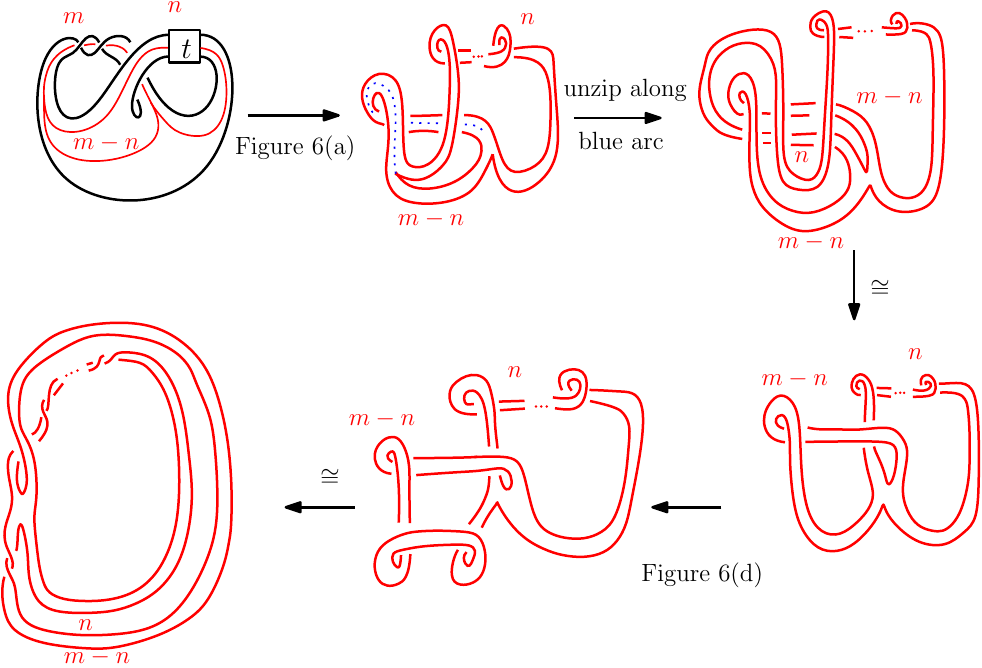}
 \caption{}
  \label{NTCase1}
\end{center}
\end{figure} 

\medskip

{\it Case 2: $(m,n)$ loop curve with $m>n>0$.}  By series isotopies, as indicated in Figure~\ref{NTCase2}, the $(m,n)$  curve in this case can be simplified to the knot depicted on the right of Figure~\ref{NTCase2}, which is the closure of negative braid in Figure~\ref{NTC}(b). 

\begin{figure}[h!]
\begin{center}
 \includegraphics[width=14cm]{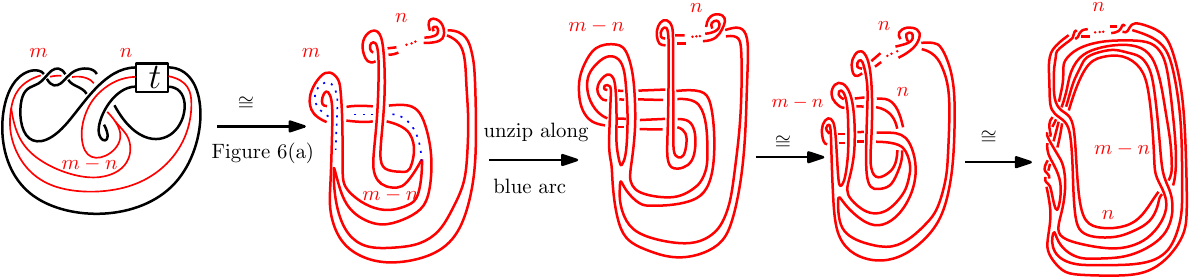}
 \caption{}
  \label{NTCase2}
\end{center}
\end{figure} 

\medskip

{\it Case 3:  $(m,n)\ \infty$ curve with $n>m>0$.} By series isotopies, as indicated in Figure~\ref{NTCase3}, the $(m,n)$  curve in this case can be simplified to the knot depicted on the bottom left of Figure~\ref{NTCase3}, which is the closure of negative braid in Figure~\ref{NTC}(c). 

\begin{figure}[h!]
\begin{center}
 \includegraphics[width=13cm]{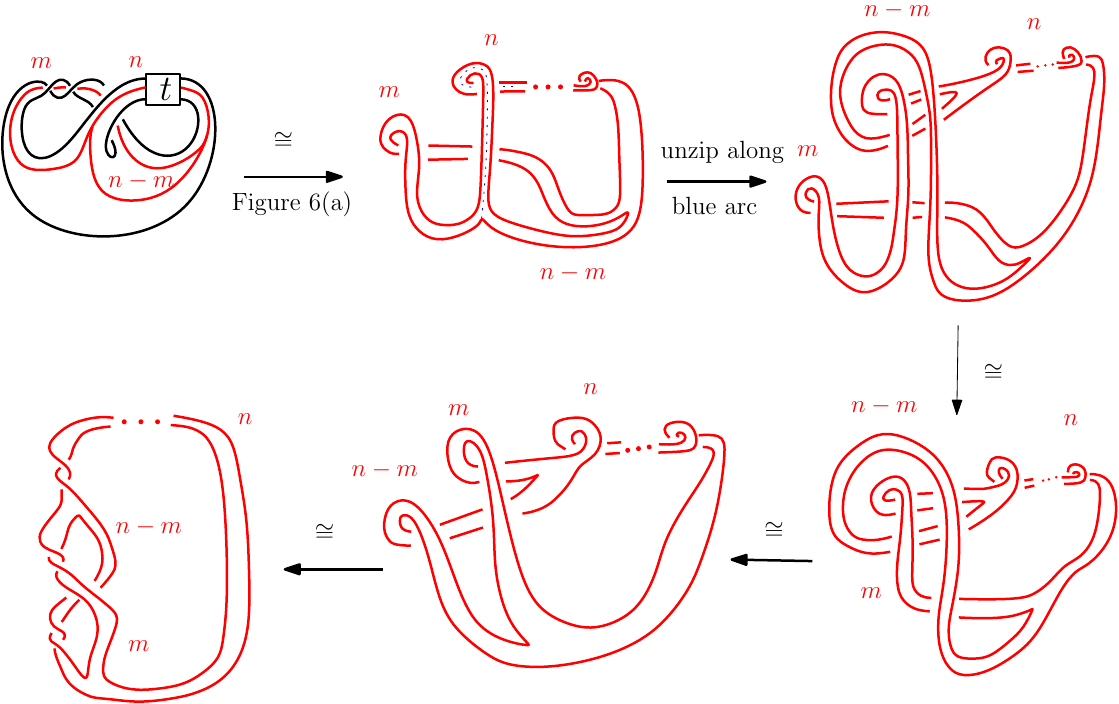}
 \caption{}
  \label{NTCase3}
\end{center}
\end{figure} 

\medskip

{\it Case 4: $(m,n)$ loop curve with $n>m>0$.} By series isotopies, as indicated in Figure~\ref{NTCase4}, the $(m,n)$  curve in this case can be simplified to the knot depicted on the right of Figure~\ref{NTCase4}, which is the closure of negative braid in Figure~\ref{NTC}(d).

\begin{figure}[h!]
\begin{center}
 \includegraphics[width=14cm]{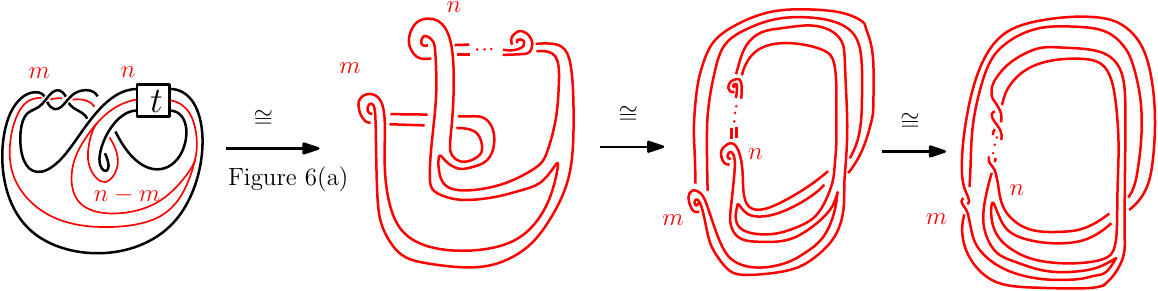}
 \caption{}
  \label{NTCase4}
\end{center}
\end{figure} 
\end{proof}

Next, we determine which of those curves in Proposition~\ref{charnegtwist} are unknotted. It is a classic result due to Cromwell \cite{Cr} (see also \cite[Corollary~$4.2$]{St}) that the Seifert algorithm applied to the closure of a positive braid gives a minimal genus surface.

\begin{proposition}\label{seifertdata} Let $\beta$ be a braid as in Figure~\ref{NTC} and $K = \hat{\beta}$ be its closure. Let $s(K)$ be the number of crossings and $l(K)$ be the number of Seifert circles Seifert circles. Then; 
\[ 
(s(K), l(K))=\begin{cases} 
      (m, |t|n(n-1) + (m - n)(m - n - 1) + n(m - n)) & ~\beta~ \textrm{as in Figure~\ref{NTC}(a)} \\
      (m+n, (|t|+1)n(n-1) + (m - n)(m - n - 1) + nm+2n(m - n)) & ~\beta~ \textrm{as in Figure~\ref{NTC}(b)} \\
      (n, (|t-1|)n(n-1) + (n - m)(n - m - 1) +m(m-1) + m(n - m)) & ~\beta~ \textrm{as in Figure~\ref{NTC}(c)} \\
      (m+n, |t|n(n-1) + m(m - 1) + nm) & ~\beta~ \textrm{as in Figure~\ref{NTC}(d)}  
   \end{cases}
\]
\end{proposition}

\begin{proof}
Consider the braid $\beta$ as in Figure~\ref{NTC}(a). Clearly, it has $m$ Seifert circles as $\beta$ has $m$ strands. Next, we will analyze the three locations in which crossings occur. First, the $t$ negative full twists on $n$ strands. Since each strand crosses over the other $n-1$ strands, we obtain $|t|n(n-1)$ crossings. Second, the negative full twist on $m-n$ strands produces additional $(m-n)(m-n-1)$ crossings. Lastly, notice the part of $\beta$ where $m-n$ strands overpass the other $n$ strands, and so for each strand in $m-n$ strands we obtain an additional $n$ crossings. Hence for $K=\hat{\beta}$ we calculate:

\[ l(\hat{\beta}) = |t|n(n-1) + (m - n)(m - n - 1) + n(m - n).\]

The calculations for the other cases are similar.

\end{proof}

We can now prove the first part of Theorem \ref{Twist}.

\begin{proof}[Proof of Theorem~\ref{Twist}, part(a)]
Proposition \ref{charnegtwist} proves the first half of our theorem. To determine there are exactly six unknotted curves when $t=-1$ and five when $t<-1$, let $B$ be the set containing the six and five unknotted curves as in Figure \ref{unknottedcurvestref} and \ref{twistunknots}, respectively. It suffices to show an essential, simple closed curve $c$ on $\Sigma_K$ where  $c \not\in B$, cannot be unknotted in $S^3.$ We know by Proposition \ref{charnegtwist}, $c$ is the closure of one of the braids in Figure~\ref{NTC} in $S^3$, where $m,n \geq 1,\ gcd(m,n) = 1$. We show, case by case, that the Seifert surface obtained via the Seifert algorithm for curves $c\not\in B$ in each case has positive genus, and hence it cannot be unknotted. 

\begin{itemize}
\item Let $c=(m,n)$ be the closure of the negative braid as in Figure~\ref{NTC}(a) and $\Sigma_{c}$ its Seifert surface obtained by the Seifert algorithm. There are $m$ Seifert circles and by Proposition~\ref{seifertdata} 
\[
l(c) = |t|n(n-1) + (m-n)(m-n-1) + n(m-n).
\]

Hence,

\[g(\Sigma_c) = \frac{1 + l - s}{2}= \frac{m(m-n-2) + n(|t|(n-1)+1) + 1}{2}.\]

If $m=n+1$, then we get $g(\Sigma_c)=\frac{|t|n(n-1)}{2}$ which is positive as long as $n>1$--note that when $c=(2,1)$ we indeed get an unknotted curve. If $m>n+1$, then $g(\Sigma_c)\geq \frac{n(|t|(n-1)+1)+1}{2}>0$ as long as $n>0$. So, $c\not\in B$ is not an unknotted curve as long as $m>n\geq 1$.

\medskip

\item Let $c=(m,n)$ be the closure of the negative braid as in Figure~\ref{NTC}(b) and $\Sigma_{c}$ its Seifert surface obtained by the Seifert algorithm. There are $n+m$ Seifert circles and by Proposition~\ref{seifertdata} 
\[l(c) = (|t|+1)n(n-1) + (m - n)(m - n - 1) + nm+2n(m - n).\]

Hence,

\[g(\Sigma_c) = \frac{m(m+n-2)+n(|t|(n-1)-1)+1}{2}.\]

One can easily see that this quantity is always positive as long as $n\geq 1$. So, $c\not\in B$ is not an unknotted curve when  $m>n\geq 1$.

\medskip

\item Let $c=(m,n)$ be the closure of the negative braid as in Figure~\ref{NTC}(c) and $\Sigma_{c}$ its Seifert surface obtained by the Seifert algorithm. There are $n$ Seifert circles and by Proposition~\ref{seifertdata} 
\[l(c) = (|t|-1)n(n-1) + (n - m)(n - m - 1) + m(m-1)+m(n - m).\]

Hence,

\[g(\Sigma_c) = \frac{n(|t|(n-1)-m-1)+m^2+1}{2}.\]

This is always positive as long as $m\geq 1$ and $|t|\neq 1$--note that when $c=(1,2)$ and $|t|=1$ we indeed get unknotted curve. So, $c\not\in B$ is not an unknotted curve when  $n>m\geq 1$.

\medskip

\item Let $c=(m,n)$ be the closure of the negative braid as in Figure~\ref{NTC}(d) and $\Sigma_{c}$ its Seifert surface obtained by the Seifert algorithm. There are $n+m$ Seifert circles and by Proposition~\ref{seifertdata} 
\[l(c) = |t|n(n-1) + m(m  - 1) + nm.\]

Hence,

\[g(\Sigma_c) = \frac{|t|n(n-1)+m(m-2)+n(m-1)+1}{2}.\]

One can easily see that this quantity is always positive as long as $m\geq 0$. So, $c\not\in B$ is not an unknotted curve when  $n>m\geq 1$.

\end{itemize}

This completes the first part of Theorem \ref{Twist}.

\end{proof}

\subsection{Figure eight knot}\label{fep} The case of figure eight knot is certainly the most interesting one. It is rather surprising, even to the authors, that there exists a genus one knot with infinitely many unknotted curves on its genus one Seifert surface. As we will see understanding homologically essential curves for the figure eight knot will be similar to what we did in the previous section. The key difference develops in Case 2 and 4 below where we show how, under certain conditions, a homologically essential $(m,n)$ $\infty$ (resp. $(m,n)$ loop) curve can be reduced to the homologically essential $(m-n, 2n-m)$ $\infty$ (resp. $(2m-n, n-m)$ loop) curve, and how this recursively produces infinitely many distinct homology classes that are represented by the unknot, and we will show that certain Fibonacci numbers can be used to describe these unknotted curves. Finally we will show fort he figure eight knot this is the only way that an unknotted curve can arise. Adapting the notations developed thus far we start characterizing homologically essential simple closed curves on genus one Seifert surface $\Sigma_K$ of the figure eight knot $K$.   

\begin{proposition} \label{char41}
All essential, simple closed curves on $\Sigma_K$ can be characterized as the closure of one of the braids in Figure~\ref{41Braids} (note the first and third braids from the left are negative and positive braids, respectively). 
\end{proposition}

\begin{figure}
    \centering
    \includegraphics[scale = .75]{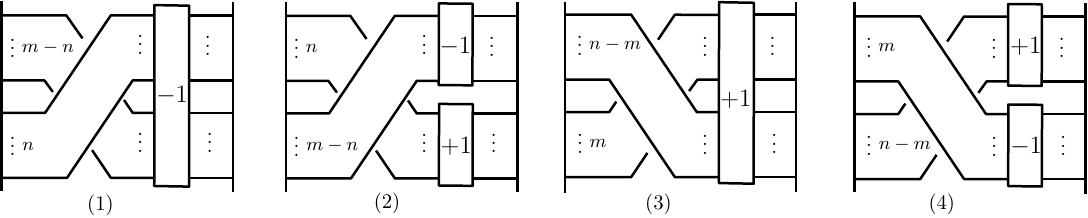}
    \caption{Braid representations of curves on $\Sigma_K$ where $K$ is the figure eight knot. From left to right:  $(m,n)$~\textrm{loop curve} with $m>n$; $(m,n)$ $\infty$ curve with $m > n$; $(m,n)$ $\infty$ curve with $n>m$ ; $(m,n)$~\textrm{loop curve} with $n>m$}
    \label{41Braids}
\end{figure}

\begin{proof}
 
The curves $(1,0)$, $(0,1)$ are clearly unknots. Moreover, because $\textrm{gcd}(m,n)=1$, the only curve with $n=m$ is $(1,1)$ curve, which is also unknot in $S^3$. For the rest of the arguments below, we will assume $n>m$ or $m>n$. There are four cases to consider:

{\it Case 1: $(m,n)$ loop curve with $m>n>0$.}

This curve can be turned into a negative braid following the process in Figure~\ref{transformingLM}.

\begin{figure}[h!]
    \centering
    \includegraphics[scale = 0.5]{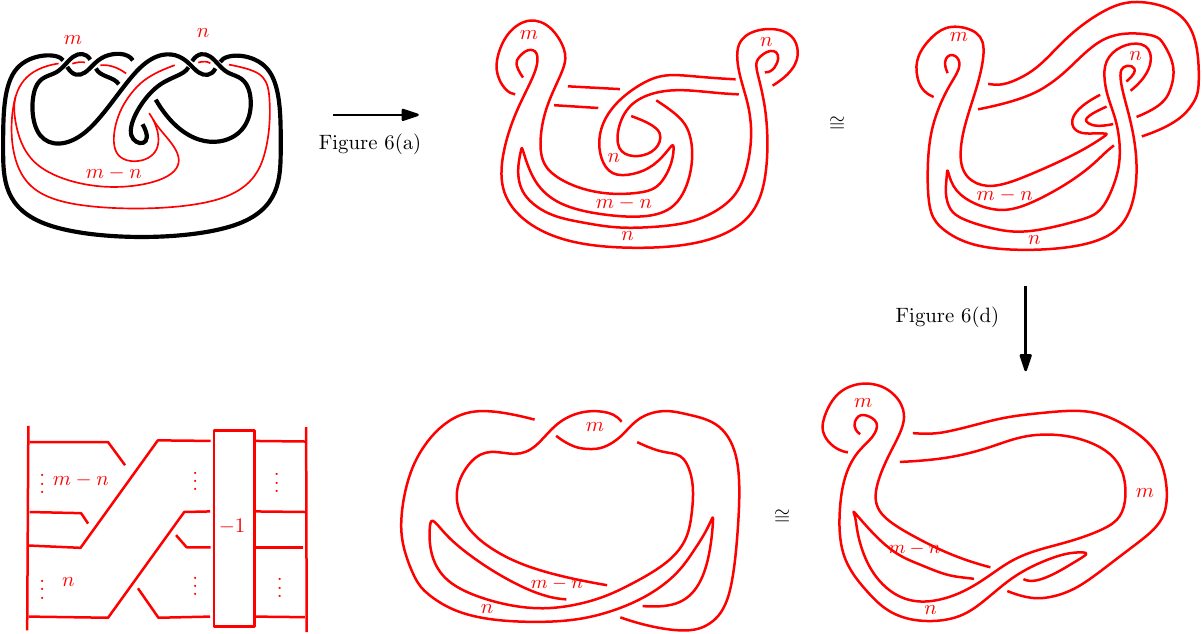}
    \caption{}
        \label{transformingLM}
\end{figure}

\medskip

{\it Case 2: $(m,n)$ $\infty$ curve with $m> n>0$.} As mentioned at the beginning, this case (and Case 4) are much more involved and interesting (in particular the subcases of Case 2c and 4c). Following the process as in Figure~\ref{transformingLN}, the curve can be isotoped as in the bottom right of that figure, which is the closure of the braid on its left--that is the second braid from the left in Figure~\ref{41Braids}.

\begin{figure}[h!]
    \centering
    \includegraphics[scale = 0.5]{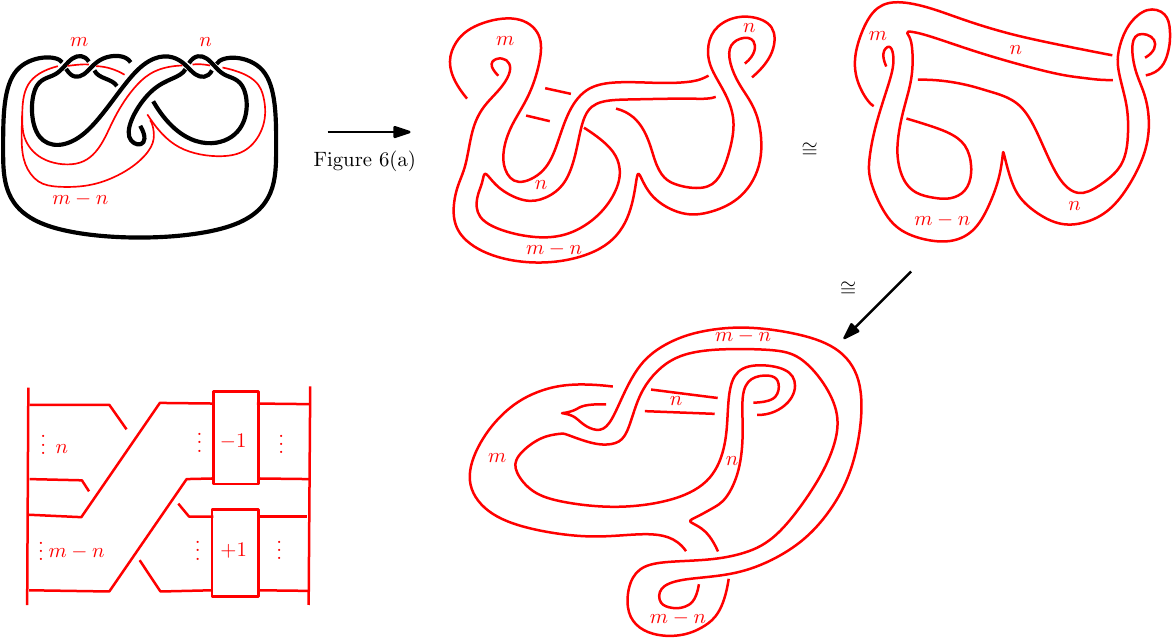}
    \caption{}
    \label{transformingLN}
\end{figure}

{\it Case 3: $(m,n)$ $\infty$ curve with $n>m>0$.} This curve can be turned into a positive braid following the process in Figure~\ref{transforming8N}.

\begin{figure}[h!]
    \centering
    \includegraphics[scale = 0.5]{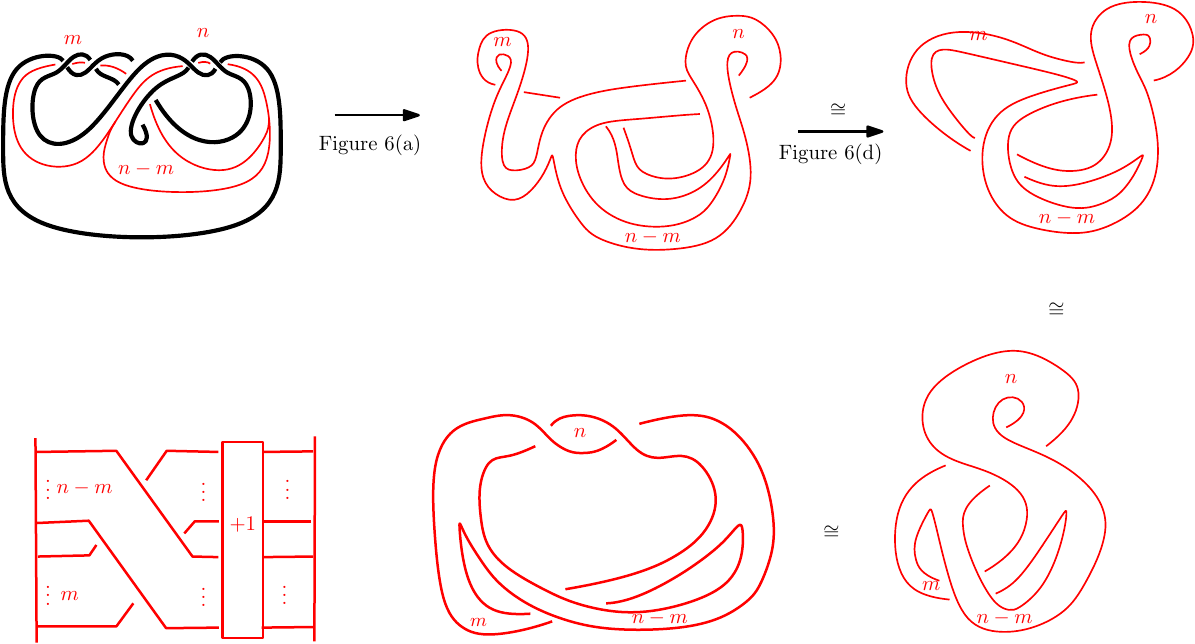}
    \caption{}
    \label{transforming8N}
\end{figure}

\medskip

{\it Case 4: $(m,n)$ loop curve with $n> m>0$}. This curve can be turned into the closure of a braid following the process in Figure~\ref{transforming8M}.

\begin{figure}[h!]
    \centering
    \includegraphics[scale = 0.5]{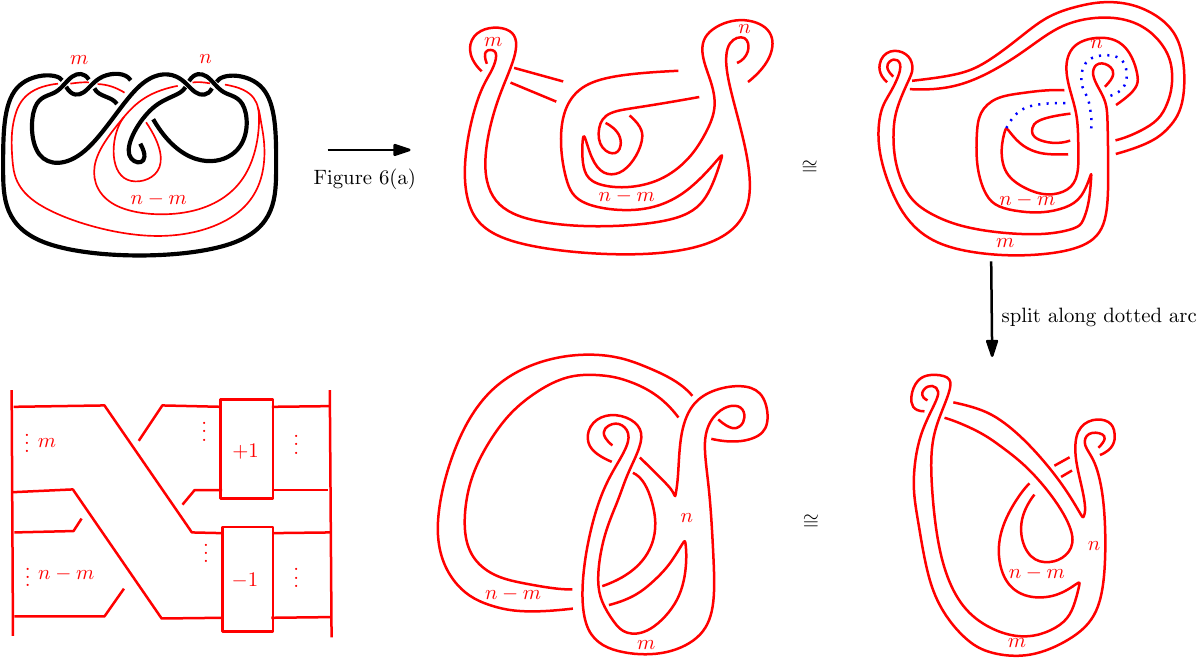}
    \caption{}
    \label{transforming8M}
\end{figure}

\end{proof}

We next determine which of these curves are unknotted:

\begin{proposition}\label{41unknots}
A homologically essential curve $c$ characterized as in Proposition~\ref{char41} is unknotted if and only if it is $(a)$ a trivial curve $(1,0)$ or $(0,1)$, $(b)$ an $\infty$ curve in the form of $(F_{i+1},F_{i})$, or $(c)$ a loop curve in the form of $(F_{i},F_{i+1})$.\\
\end{proposition}

\begin{proof}
Let $c$ denote one of these homologically essential curve listed in Proposition~\ref{char41}. We will analyze the unknottedness of $c$ in four separate cases.

{\it Case 1.} Suppose $c=(m,n)$ is the closure of the negative braid in the bottom left of Figure~\ref{transformingLM}. Note the minimal Seifert Surface of c, $\Sigma_c$, has $(n)(m-n)+(m)(m-1)$ crossings and $m$ Seifert circles. Hence;
\[g(\Sigma_c) = \frac{n(m-n)+(m-1)^2}{2}\]
This is a positive integer for all $m,n$ with $m>n$. So $c$ is never unknotted in $S^3$ as long $m>n>0$ .

{\it Case 2.} Suppose $c$ is of the form in the bottom right of Figure~\ref{transformingLN}. Since this curve is not a positive or negative braid closure, we cannot directly use Cromwell's result as in Case 1 or the previous section. There are three subcases to consider.

{\it Case 2a: $m-n=n$.} Because $m$ and $n$ are relatively prime integers, we must have that $m=2, n=1$, and we can easily see that this $(2,1)$ curve unknotted.

{\it Case 2b: $m-n>n$.} This curve can be turned into a negative braid following the process in Figure~\ref{transforming2b}.  More precisely, we start, on the top left of that figure, with the curve appearing on the bottom right of Figure~\ref{transformingLN}. We extend the split along the dotted blue arc and isotope $m$ strands to reach the next figure. We note that this splitting can be done as by the assumption we have $m-2n>0$. Then using Figure~\ref{SC}(a) and further isotopy we reach the final curve on the bottom right of Figure~\ref{transforming2b} which is obviously the closure of the negative braid depicted on the bottom left of that picture.

\begin{figure}[h!]
    \centering
    \includegraphics[scale = 0.5]{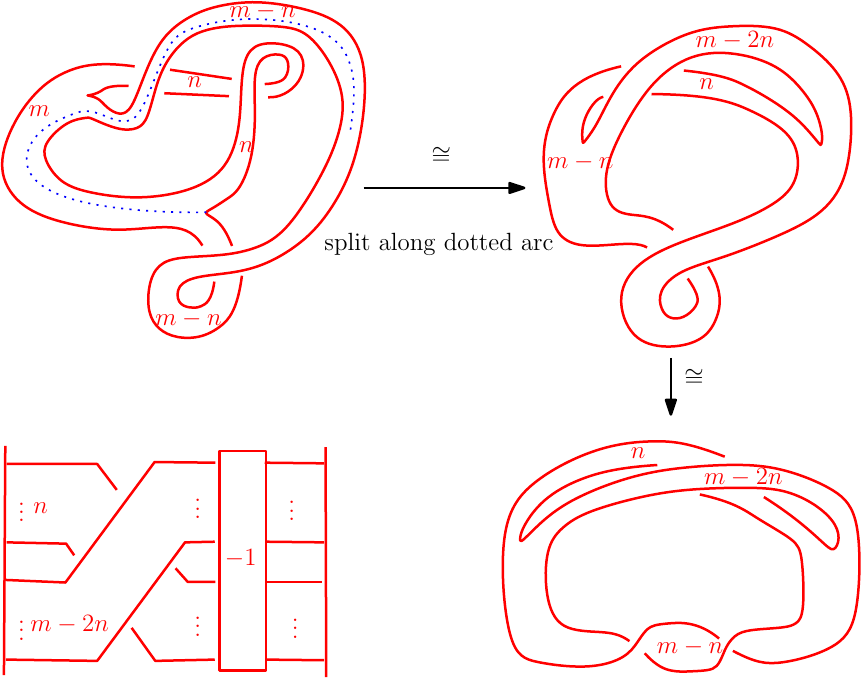}
    \caption{}
    \label{transforming2b}
\end{figure}

The minimal Seifert Surface coming from this negative braid closure contains $m-n$ circles and $(m-2n)n+(m-n)(m-n-1)$
twists. Hence;

\[g(\Sigma_c)=\frac{(m-2n)n+(m-n)(m-n-2)+1}{2}.\]

This a positive integer for all integers $m,n$ with $m-n>n$. So, $c$ is not unknotted in $S^3$.

{\it Case 2c: $m-n<n$.} We organize this curve some more. We start, on the top left of Figure~\ref{transforming2c}, with the curve that is appearing on the bottom left of Figure~\ref{transformingLN}. We extend the split along the dotted blue arc and isotope $m-n$ strands to reach the next figure, After some isotopies we reach the curve on the bottom left of Figure~\ref{transforming2c}. In other words, this subcase of Case 2c leads to a reduced version of the original picture (top left curve in Figure~\ref{transformingLN}), in the sense that the number of strands over either handle is less than the number of strands in the original picture.   

\begin{figure}[h!]
    \centering
    \includegraphics[scale = 0.5]{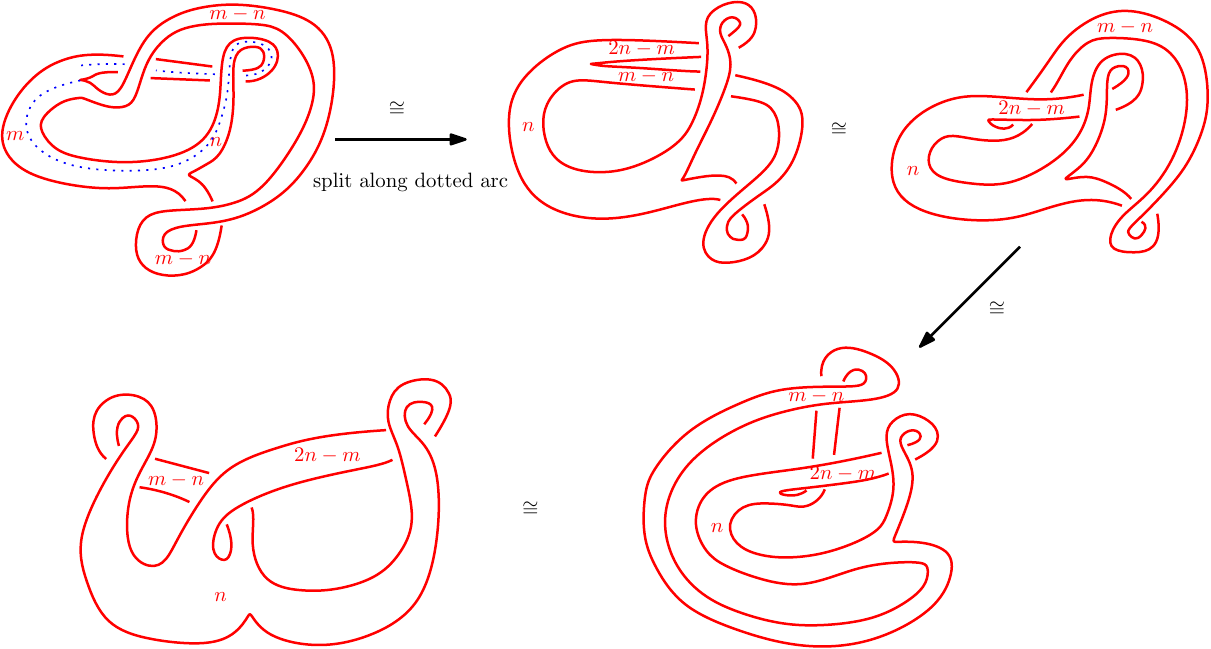}
    \caption{}
    \label{transforming2c}
\end{figure}

This case can be further subdivided depending on the relationship between $2n-m$ and $m-n$, but this braid (or rather its closure) will turn into a $(m-n, 2n-m)$ $\infty$ curve when $m-n>2n-m$:

\underline{\bf Case 2c-i: $2n-m = m-n$.} This simplifies to $3n=2m$. Because $gcd(m,n)=1$, this will only occur for $m=3$ and $n=2$, and the resulting curve is $(1,1)$ $\infty$ curve. In other words here we observed that $(3,2)$ curve has been reduced to (1,1) curve

\underline{\bf Case 2c-ii: $2n-m > m-n$.} This means that we are dealing with a curve under Case 3, and we will see that all curves considered there are positive braid closures.

\underline{\bf Case2c-iii: $2n-m < m-n$.}  This means we are back to be under Case 2. So for $m>n>m-n$, the $(m,n)$ $\infty$ curve is isotopic to the $(m-n, 2n-m)$ $\infty$ curve. This isotopy series will be notated $(m,n) \sim (m-n, 2n-m)$. Equivalently, there is a series of isotopies such that $(m-n, 2n - m) \sim (m,n)$. If $(k,l)$ denote a curve at one stage of this isotopy, then $(k,l) \sim ((k +l) + k, k+l)$. So, starting with $k = l = 1$, we recursively obtain:
\[(1,1) \sim (3,2) \sim (8,5) \sim (21, 13) \sim (55, 34) \sim \cdots\]

In a similar fashion, if we start with $k = 2,\ l = 1$ we obtain:
\[(2,1) \sim (5,3) \sim (13,8) \sim (34, 21) \sim (89, 55) \sim \cdots\]

Notice every curve $c$ above is of the form $c = (F_{i + 1}, F_i), i \in \mathbb{Z}_{>0}$ where $F_i$ denotes the $i^{th}$ {\it Fibonacci number}. We will call these {\it Fibonacci curves}. We choose $(1,1)$ and $(2,1)$ because they are known unknots. As a result, this relation generates an infinite family of homologically distinct simple closed curves on $\Sigma_K$ that are unknotted in $S^3$.
 
{\it Case 3}. Suppose a curve, $c$, is of the form $(3)$, which is the closure of the positive braid depicted in the bottom left of Figure~\ref{transforming8N}.  An argument similar to that applied to Case 1 can be used to show $c$ is never unknotted in $S^3.$ 

\medskip

{\it Case 4}. Suppose $c$ is of the form as in the bottom middle of Figure~\ref{transforming8M}. Similar to Case 2, there are three subcases to consider.

\medskip

{\it Case 4a:  $m = n - m$.} Then $2m=n$. Because $gcd(m,n)=1$, $m=1$ and $n=2$, resulting in unknot.

{\it Case 4b: $n-m>m$.} Then $n-2m>0$ and following the isotopies in Figure~\ref{transforming4b}, the curve can be changed into the closure of positive braid depicted on the bottom right of that figure.

\begin{figure}[h!]
    \centering
    \includegraphics[scale = 0.5]{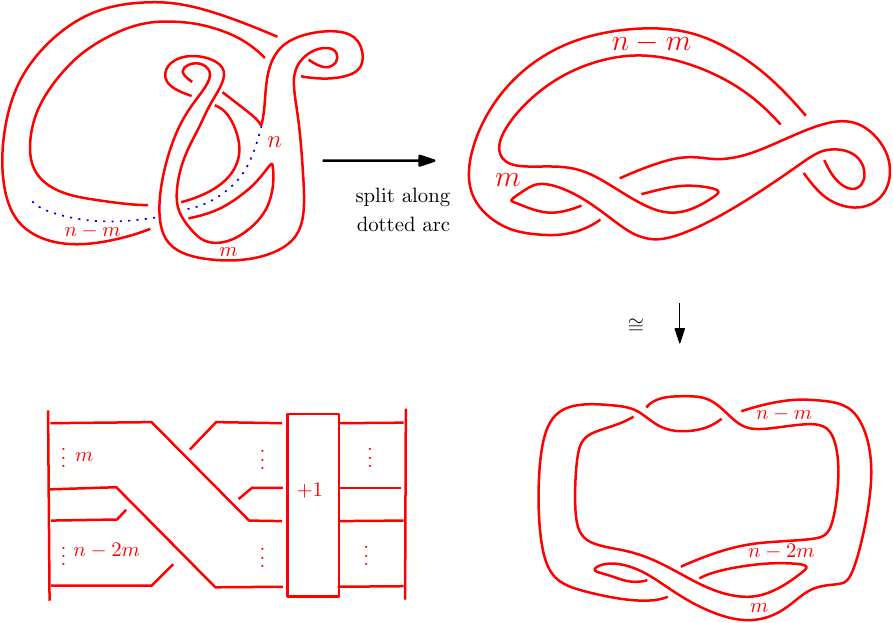}
    \caption{}
    \label{transforming4b}
\end{figure}

Identical to Case 2b, the curve $c$ in this case is never unknotted in $S^3.$

{\it Case 4c: $m>n-m$}. Then $2m-n>0$, and we can split the $m$ strands into two: a $n-m$ strands and a $2m-n$ strands.

\begin{figure}[h!]
    \centering
    \includegraphics[scale = 0.75]{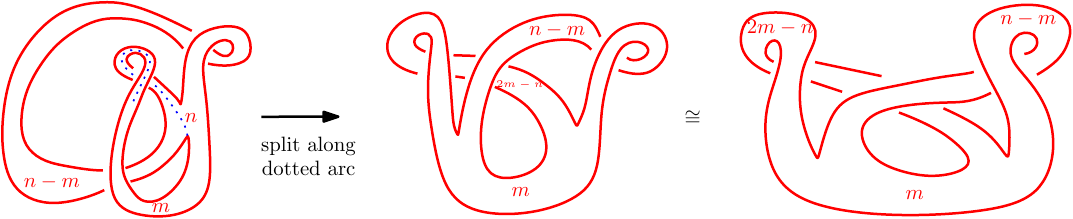}
    \caption{}
    \label{transforming4c}
\end{figure}

This case can be further subdivided depending on the relationship between $n-m$ and $2m-n$, but this braid will turn into a $(2m-n, n-m)$ loop curve when $n-m>2m-n$:

\underline{\bf Case 4c-i: $2m-n = n-m$.}
This simplifies to $3m=2n$. Because $gcd(m,n)=1$, this will only occur for $m=2$ and $n=3$, and the resulting curve is a $(1,1)$ loop curve.

\underline{\bf Case 4c-ii: $n-m < 2m-n$.}  This means that we are dealing with a curve under Case 1, and we saw that all curves considered there are negative braid closures.

\underline{\bf Case 4c-iii: $n-m > 2m-n$.} This means that we are back to be under Case 4. So for $n>m>n-m$, an $(m,n)$ loop curve has the following isotopy series: $(m,n) \sim (2m-n,n-m)$. If $(k,l)$ denote a curve at one stage of this isotopy, then the reverse also holds: $(k,l) \sim (k+l, (k+l)+l)$. As a result, much like Case 2c, we can generate two infinite families of unknotted curves in $S^3$:
\[(1,1) \sim (2,3) \sim (5,8) \sim (13, 21) \sim (34, 55) \sim \cdots ~\textrm{and} \]
\[(1,2) \sim (3,5) \sim (8,13) \sim (21, 34) \sim (55, 89) \sim \cdots\]

Notice every curve $c$ is of the form $c = (F_{i}, F_{i + 1}), i \in \mathbb{Z}_{>0}.$ Finally, we show that this is the only way one can get unknotted curves. That is, we claim:

\begin{lemma} 
If a homologically essential curve $c$ on $\Sigma_K$ for $K=4_1$ is unknotted, then it must be a Fibonacci curve. 
\end{lemma}
\begin{proof} From above, it is clear that if our curve c is Fibonacci, then it is unknotted. So it suffices to show if a curve is not Fibonacci then it is not unknotted. We will demonstrate this for loop curves under Case 4. Let $c$ be a loop curve that is not Fibonacci but is unknotted. Since it is unknotted, it fits into either Case~4a or 4c. But the only unknotted curve from Case 4a is $(1,1)$ curve which is a Fibonacci curve, so $c$ must be under Case 4c. By our isotopy relation, $(m,n) \sim (2m-n, n - m)$. So, the curve can be reduced to a minimal form, say $(a,b)$ where $(a,b) \neq (1,1)$ and  $(a,b) \neq (2,1).$ We will now analyze this reduced curve $(a,b)$: 

\begin{itemize}
\item If $a = b$, then $(a,b) = (1,1)$; a contradiction.

\item If $a > b$, then $(a,b)$ is under Case 1; none of those are unknotted.

\item If $b - a < a < b$, then $(a,b)$ is still under Case 4c, and not in reduced form; a contradiction.

\item If $a < b - a < b$, then $(a,b)$ is under Case 4b; none of those are unknotted.

\item If $b-a = a<b$, then $(a,b) = (2,1)$; a contradiction.

\end{itemize}

So, it has to be that either $(a,b) \sim (1,1)$ or $(a,b) \sim (2,1)$. Hence, it must be that $c = (F_{i}, F_{i+1})$ for some $i$. The argument for the case where $c$ is an $\infty$ curve under Case 2 is identical.

\end{proof}
\end{proof}

\subsection{Twist knot with $t>1$--Part 1}\label{ptp} In this section we consider twist knot $K=K_t$, $t\geq 2$, and give the proof of Theorem~\ref{Twist2}.

\begin{proposition} \label{charpostwist}
All essential, simple closed curves on $\Sigma_K$ can be characterized as the closure of one of the braids in Figure~\ref{TTC}.
\end{proposition}

\begin{figure}[h!]
\begin{center}
 \includegraphics[width=11cm]{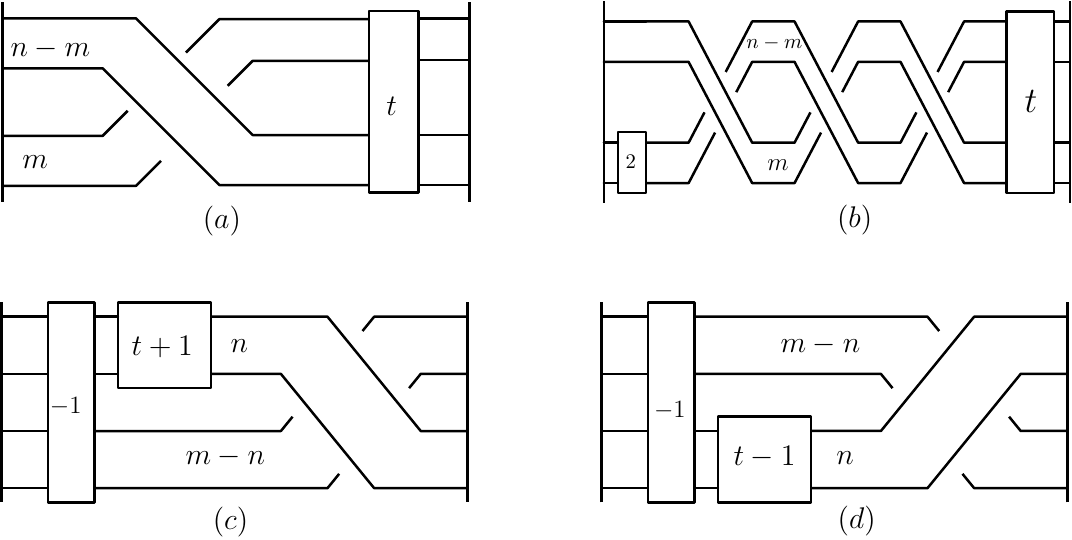}
 \caption{}
  \label{TTC}
\end{center}
\end{figure}

\begin{proof} It suffices to show all possible curves for an arbitrary $m$ and $n$ such that $gcd(m, n) = 1$ are
the closures of braids in Figure~\ref{TTC}. Here too there are four cases to consider but we will analyze these in slightly different order than in the previous two sections.

{\it Case 1:  $(m,n)\ \infty$ curve with $n>m>0$.} In this case the curve is the closure of a positive braid, and this is explained in Figure~\ref{TTCase1} below. More precisely, we start with the curve which is drawn in the top left of the figure, and after a sequence of isotopies this becomes the curve in the bottom right of the figure which is obviously the closure of the braid in the bottom left of the figure.  In particular, when $n>m\geq 1$, none of these curves will be unknotted. 

\begin{figure}[h!]
\begin{center}
 \includegraphics[width=11cm]{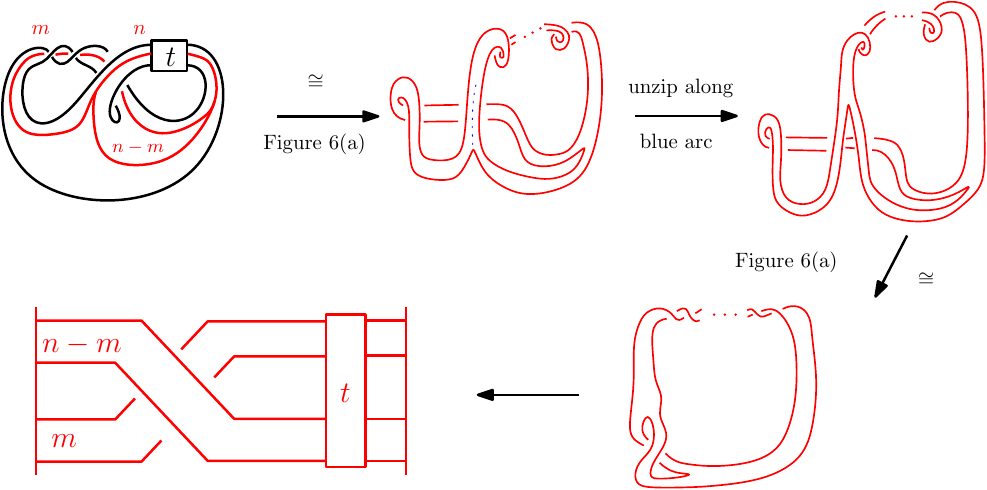}
 \caption{}
  \label{TTCase1}
\end{center}
\end{figure}

{\it Case 2:  $(m,n)$ loop curve with $n>m>0$.} In this case too the the curve is the closure of a positive braid, and this is explained in Figure~\ref{TTCase2} below. In particular, when $n>m>1$, none of these curves will be unknotted. 

\begin{figure}[h!]
\begin{center}
 \includegraphics[width=13cm]{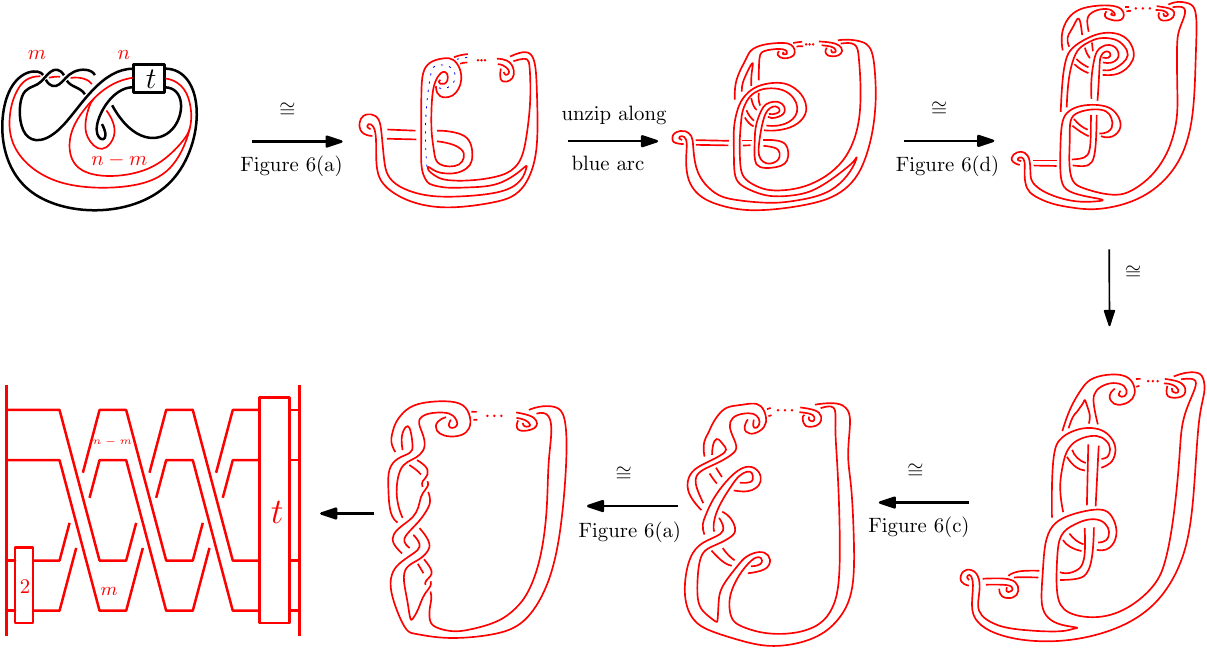}
 \caption{}
  \label{TTCase2}
\end{center}
\end{figure} 

In the remaining two cases we will follow slightly different way of identifying our curves as braid closures. As we will see (which is evident in part (c) and (d) of Proposition~\ref{charpostwist}) that the braids will not be positive or negative braids for general and $m, n$ and $t$ values. We will then verify how under the various hypothesis listed in Theorem~\ref{Twist2} these braids can be reduced to a positive or negative braids.

{\it Case 3:  $(m,n)\ \infty$ curve with $m>n>0$.} We explain in Figure~\ref{TTCase3N} below how the $(m,n)\ \infty$ curve with $m>n>0$ is the closure of the braid in the bottom left of the figure. This braid is not obviously a positive or negative braid.

\begin{figure}[h!]
\begin{center}
 \includegraphics[width=12cm]{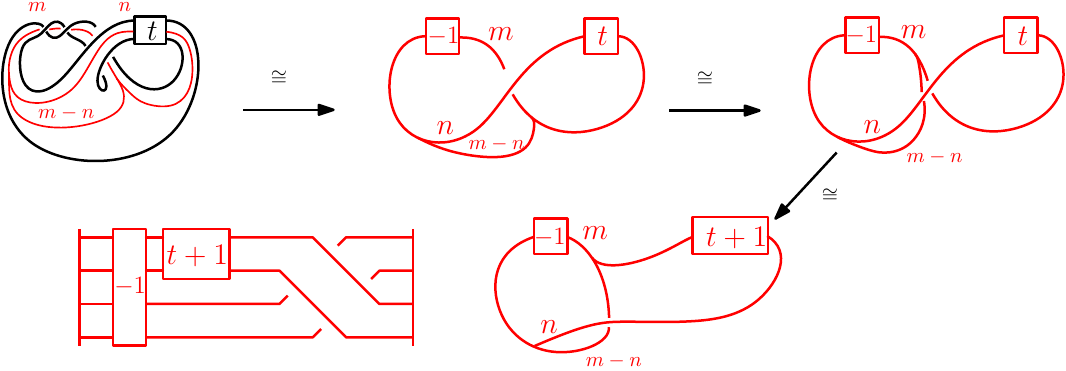}
 \caption{}
  \label{TTCase3N}
\end{center}
\end{figure} 

{\underline{\it Case 3a $(m,n)\ \infty$ curve with $m>n>0$ and $m-tn>0$}}. We want to show the braid in the bottom left of Figure~\ref{TTCase3N} under the hypothesis that $m-tn>0$ can be made a negative braid. We achieve this in Figure~\ref{TTCase3NN}. More precisely, in part (a) of the figure we see the braid that we are working on. We apply the move in Figure~7(f) and some obvious simplifications to reach the braid in part (d). In part (e) of the figure we re-organize the braid: more precisely, since $m-tn>0$ and $m-n=m-tn+(t-1)n$, we can split the piece of the braid in part (d) made of $m-n$ strands as the stack of $m-tn$ strands and set of $t-1$ $n$ strands. We then apply the move in Figure~\ref{SC}(f) repeatedly ($t-1$ times) to obtain the braid in part (f). We note that the block labeled as ``all negative crossings'' is not important for our purpose to draw explicitly but we emphasize that each time we apply the move in Figure~\ref{SC}(f) it produces a full left handed twist between an $n$ strands and the rest. Next, sliding $-1$ full twists one by one from $n$ strands over the block of these negative crossings we reach part (g). After further obvious simplifications and organizations in parts (h)--(j) we reach the braid in part (k) which is a negative braid.   

\begin{figure}[h!]
\begin{center}
 \includegraphics[width=14cm]{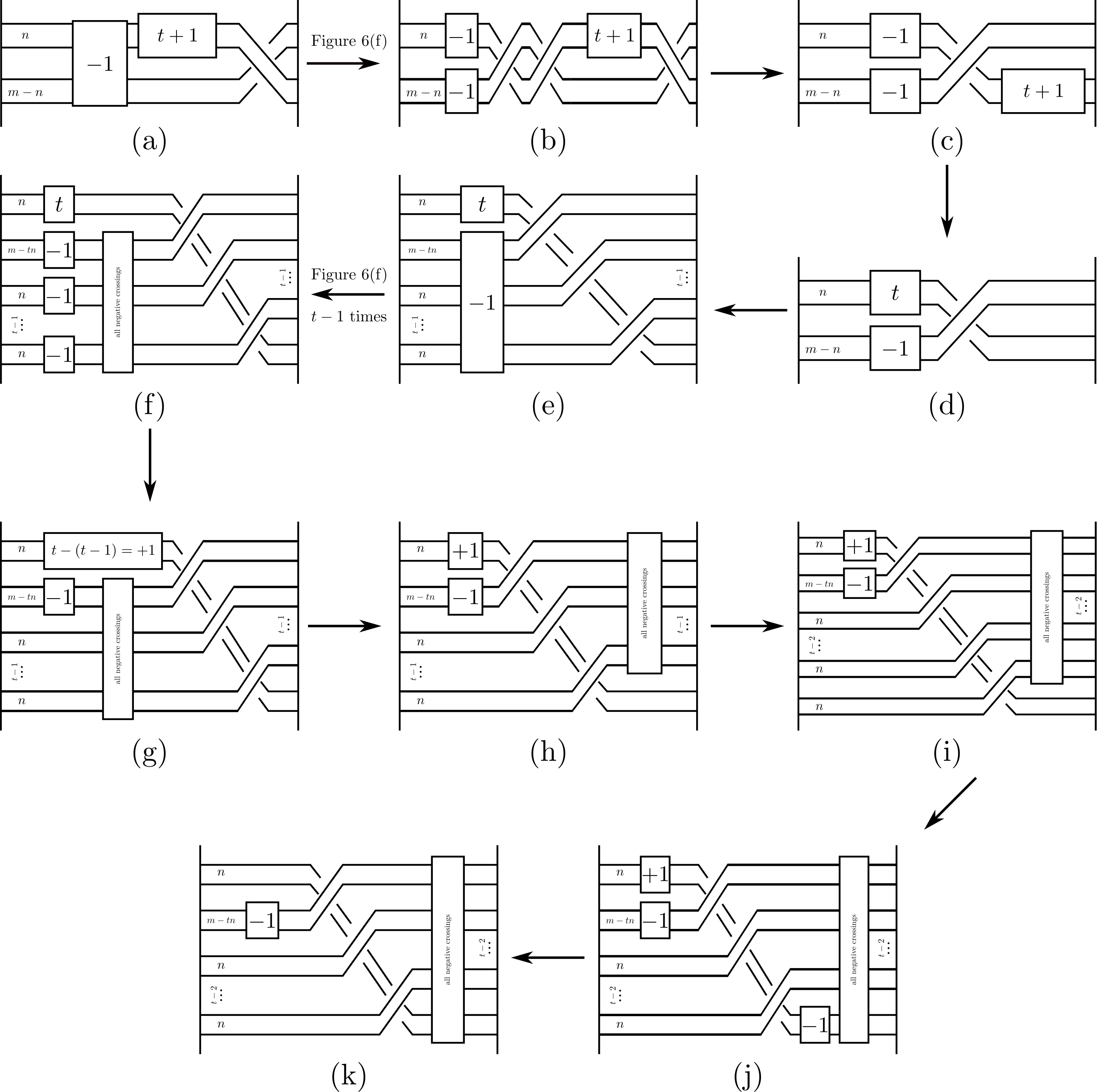}
 \caption{}
  \label{TTCase3NN}
\end{center}
\end{figure} 

{\underline{\it Case 3b $(m,n)\ \infty$ curve with $m>n>0$ and $m-n<n$}}. We want to show in this case the braid in the bottom left of Figure~\ref{TTCase3N} under the hypothesis that $m-n<n$ can be made a positive braid (regardless of $t$ value). This is achieved in Figures ~\ref{Case3sporadic}.  

\begin{figure}[h!]
\begin{center}
 \includegraphics[width=14cm]{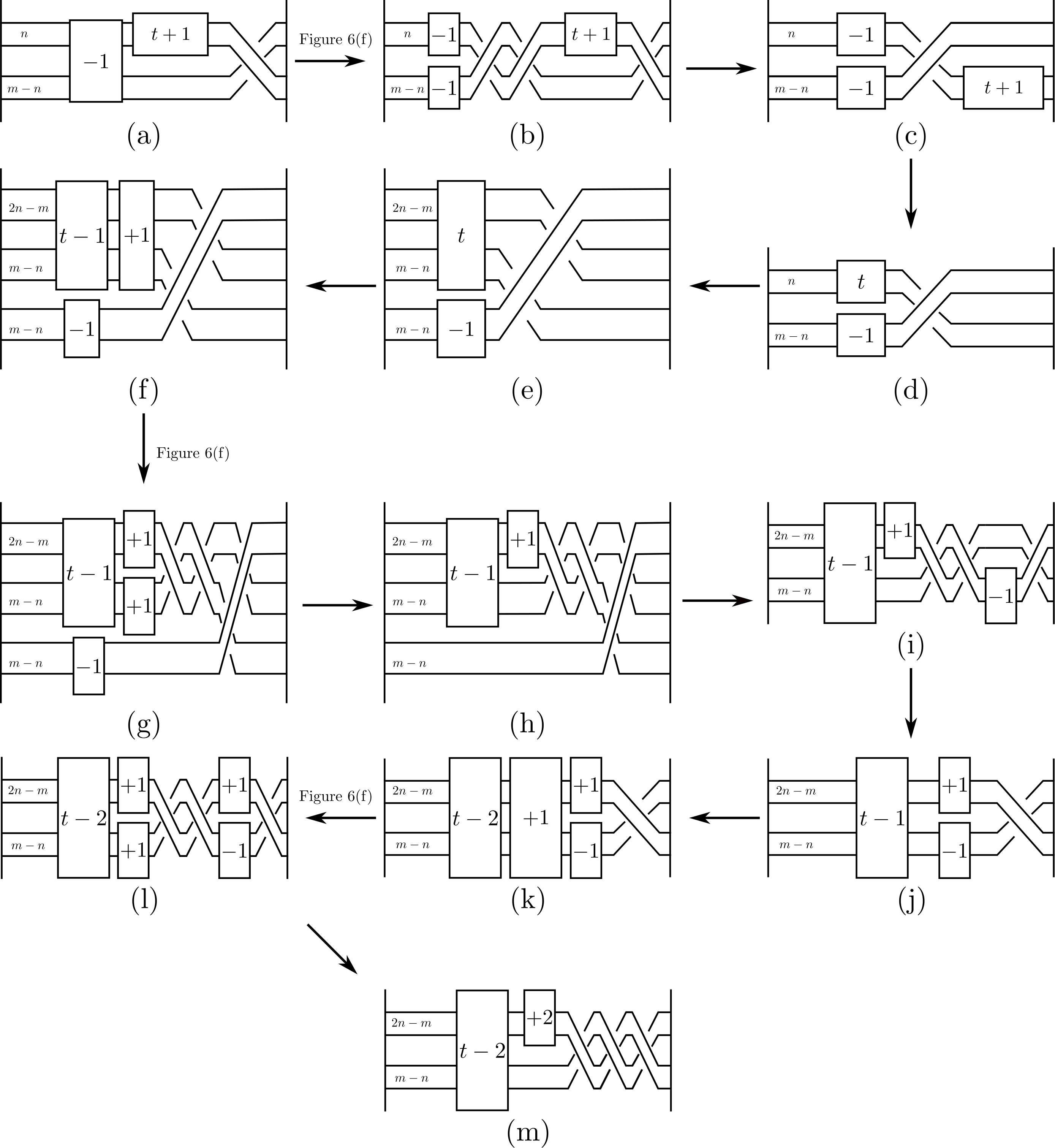}
 \caption{}
  \label{Case3sporadic}
\end{center}
\end{figure}

{\it Case 4:  $(m,n)$ loop curve with $m>n>0$.} The arguments for this case are identical Case 3 and  3a above. The $(m,n)$ loop curve with $m>n>0$ is the closure of the braid that is drawn in the bottom left of Figure~\ref{TTCase4N}. 

\begin{figure}[h!]
\begin{center}
 \includegraphics[width=12cm]{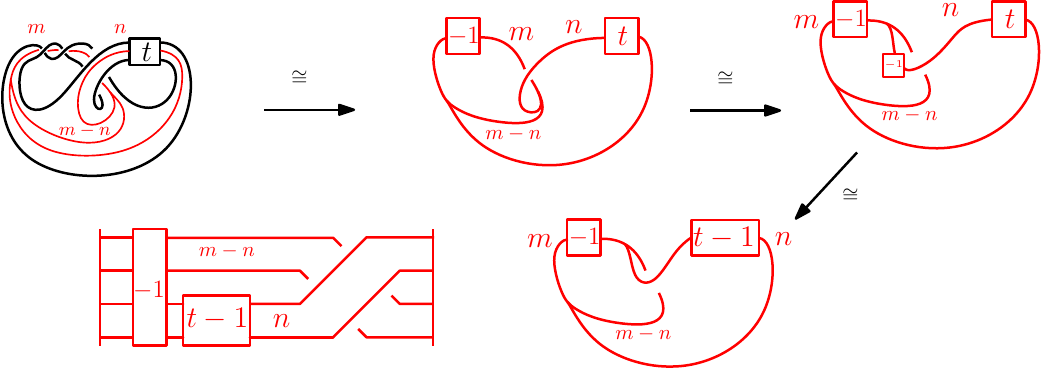}
 \caption{}
  \label{TTCase4N}
\end{center}
\end{figure}

{\underline{\it Case 4a $(m,n)$ loop curve with $m>n>0$ and $m-tn>0$}}. We show the braid, which the $(m,n)\ \infty$ curve with $m>n>0$ is closure of, can be made a negative braid under the hypothesis $m-tn>0$. This follows very similar steps as in Case 3a which is explained through a series drawings in Figure~\ref{TTCase4NN}.  

\begin{figure}[h!]
\begin{center}
 \includegraphics[width=14cm]{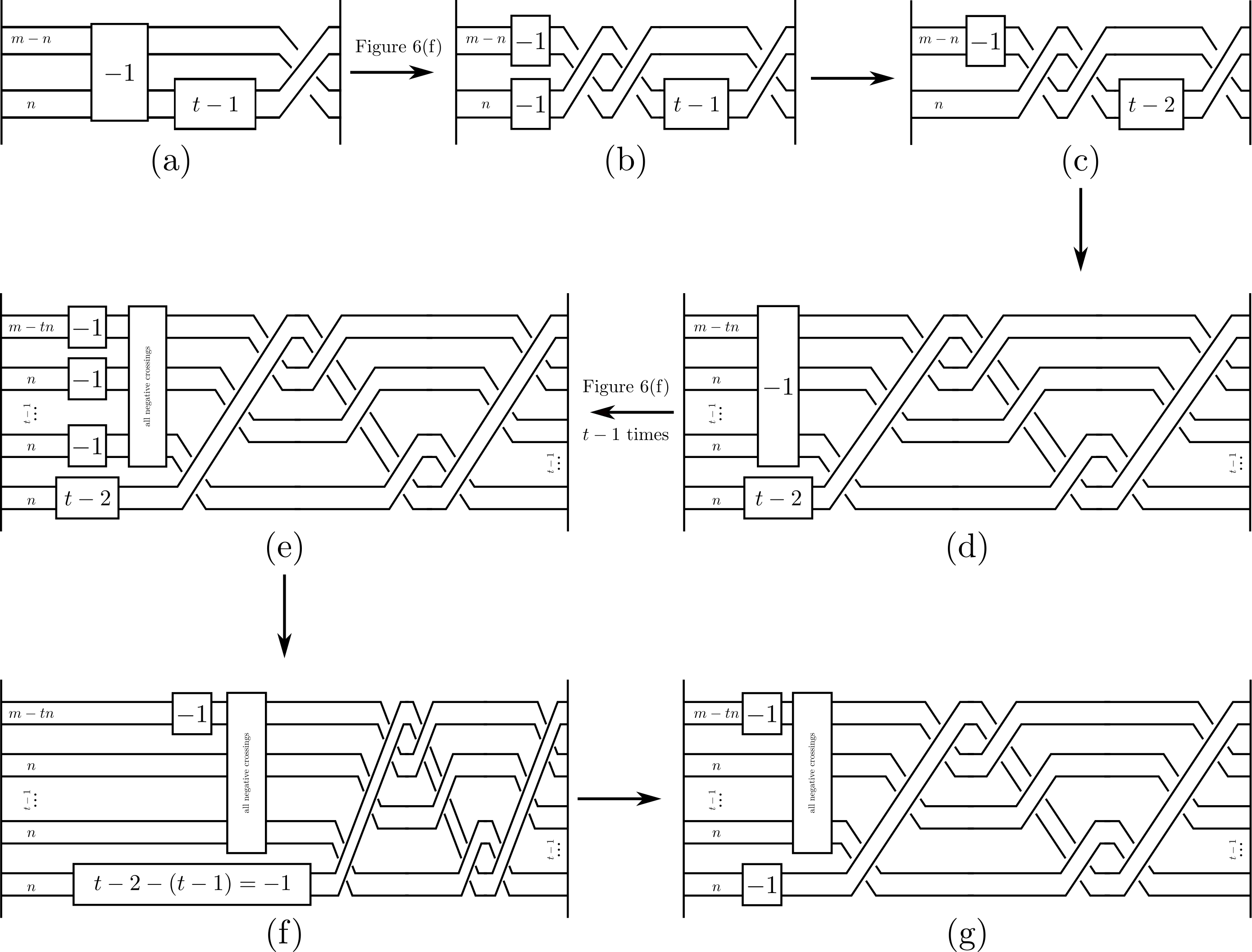}
 \caption{}
  \label{TTCase4NN}
\end{center}
\end{figure}

{\underline{\it Case 4b $(m,n)$ loop curve with $m>n>0$ and $m-n<n$}}. Finally, we consider the $(m,n)$ loop curve with $m>n>0$ and $m-n<n$. Interestingly, this curve for $t>2$ does not have to the closure of a positive or negative braid. This will be further explored in the next section but for now we observe, through Figure~\ref{TTCase4NN}(a)-(c) that when $t=2$ the curve is the closure of a negative braid: The braid in (a) in the figure is the braid from Figure~\ref{TTC}(d). After applying the move in Figure~\ref{SC}{f}, and simple isotopies we obtain the braid in (c) which is clearly a negative braid when $t=2$.

\end{proof}

\begin{proof}[Proof of Theorem~\ref{Twist2}]
The proof of part (1) follows from Case 1 and 2 above. Part (2)a/b follows from Case 3a/b and Case 4a above. As for part (3), observe that when $n>m$ by using Case 1 and 2 we obtain that all homologically essential curves are the closures of positive braids. When $m>n$, we have either $m-2n>0$ or $m-2n<0$. In the former case we use Case 3a and 4a to obtain that all homologically essential curves are the closures of negative braids. In the latter case, first note that $m-2n<0$ is equivalent to $m-n<n$, Now by Case 3b all homologically essential $\infty$ curves are the closures of positive braids, and by Case 4b all homologically essential loop curves are the closures of negative braids. Now by using Cromwell's result and some straightforward genus calculations we deduce that when $m>n>1$ or $n>m\geq 1$ there are no unknotted curves among (positive/negative) braid closures obtained in Case $1-4$ above. Therefore, there are exactly 5 unknotted curves among homologically essential curves on $\Sigma_K$ for $K=K_t$ in Theorem~\ref{Twist2}. 
\end{proof}

\subsection{Twist knot with $t>1$--Part 2}\label{ptp2} In this section we consider twist knot $K=K_t$, $t\geq 3$, and give the proof of Theorem~\ref{Twist3}.

\begin{proof}[Proof of Theorem~\ref{Twist3}]
We show that the loop curve $(3,2)$ when $t\geq 3$ is the pretzel knot $P(2t-5, -3,2)$. This is explained in Figure~\ref{PosTwistSpec}. The braid in $(a)$ is from Figure~\ref{TTC}(d) with $m=3, n=2$, where we moved $(t-2)$ full right handed twists to the top right end. We take the closure of the braid and cancel the left handed half twist on the top left with one of the right handed half twists on the top right to reach the knot in $(c)$. In $(c)-(g)$ we implement simple isotopies, and finally reach, in $(h)$, the pretzel knot $P(2t-5, -3,2)$. This knot has genus $t-1$ (\cite{KimLee}[Corollary~2.7] , and so is never unknotted as long as $t>1$. This pretzel knot is slice exactly when $2t-5+(-3)=0$. That is when $t=4$. The pretzel knot $P(3,-3,3)$ is also known as $8_{20}$. An interesting observation is that although $P(2t-5, -3, 2)$ for $t>2$ is not a positive braid closure, it is a quasi-positive braid closure.

\begin{figure}[h!]
\begin{center}
 \includegraphics[width=13cm]{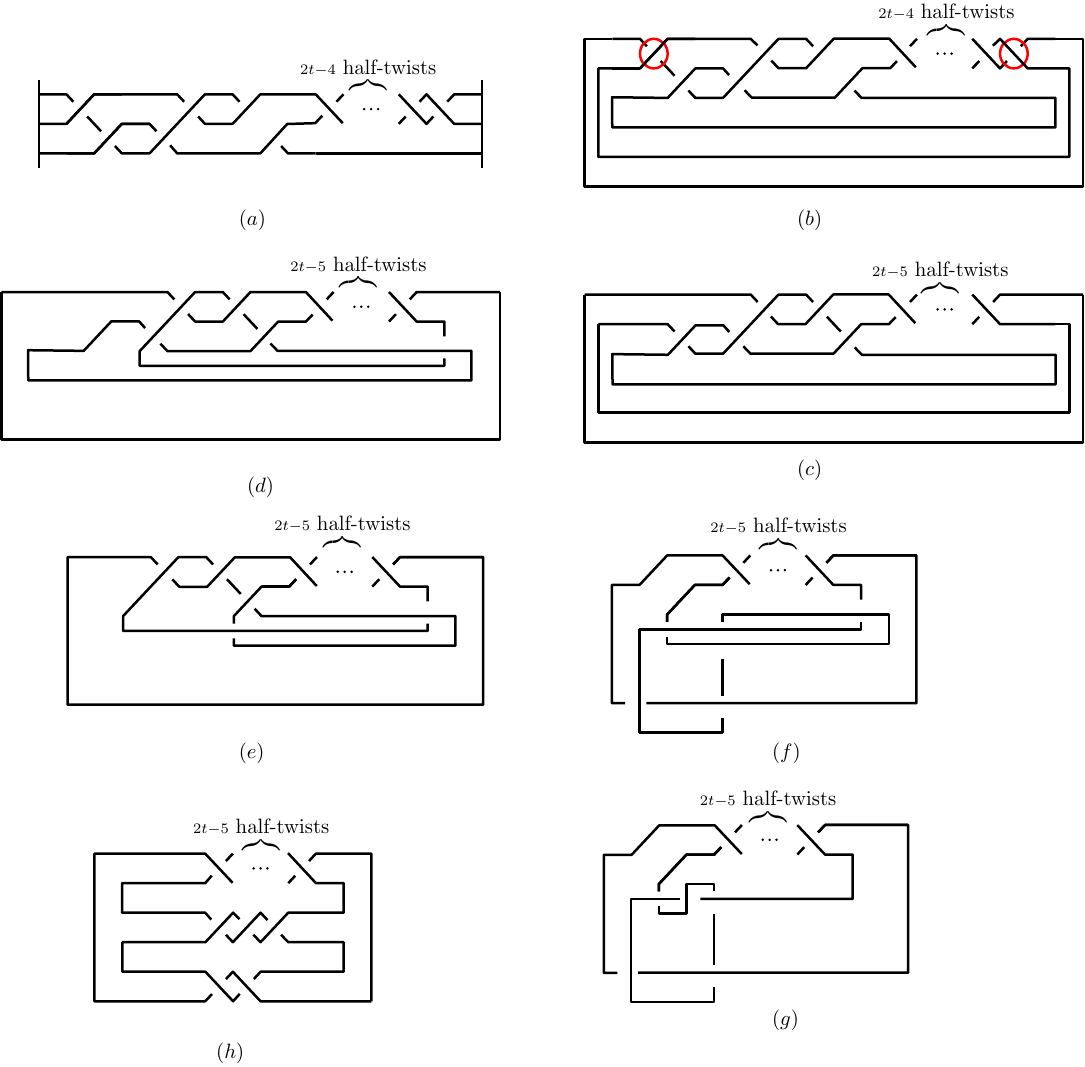}
 \caption{}
  \label{PosTwistSpec}
\end{center}
\end{figure} 

\end{proof}

\begin{proposition}\label{Rtwist}
The $(m,n)$ loop curve with $m-n=1$, $n>3$ and $t> 4$ is never slice.
\end{proposition}
\begin{proof} 
By Rudoplh in \cite{rudolph}, we have that for a braid closure $\hat{\beta}$ when $k_{+} \neq k_{-}$
\[g_4(\hat{\beta}) \geq \frac{|k_{+} - k_{-}| - n + 1}{2}\]
 where $\beta$ is a braid in $n$ strands, and $k_{\pm}$ is the number of positive and negative crossings in $\beta$. For quasi-positive knots, equality holds. In which case, the Seifert genus is also the same as the four ball (slice) genus.

Note that this formula can also be thought as a generalization to the Seifert genus calculation formula we used for positive/negative braid closures, since for those braids when, $|k_{+} - k_{-}|$ is the number of crossings and $n$, the braid number, is exactly the number of Seifert circles. Thus Rudoplh's inequality can also be used to state that the above calculations to rule out unknotted curves on various genus one Seifert surface can also be used to state that there are no slice knots other than the unknotted ones found.

Now for the loop curve $c=(m,n)$ as in Figure~\ref{TTCase4NN}(c), we have that 
\[k_{+} = (t-2)n(n-1), k_{-} = (m-n)(m-n-1) + 3(m-n)n \]

Hence, when $m-n = 1$, we get that $k_{-} = 3n$. Notice also that for $n \geq 3, t \geq 4$, we have $k_{+} > k_{-}$. Thus, for $n > 3, t > 4, m-n=1$ we obtain $c=\hat{\beta}$ is never slice as;

\[g_4(\hat{\beta}=c) \geq \frac{(t-2)n(n-1) - 3n - m +1}{2} = n((t-2)(n-1) - 4) > 0\]

It can be manually checked that the $(4,3)$ loop curve when $t = 3$ is not slice either.
\end{proof}

\section{Whitehead Doubles}\label{white}
In this section we provide the proof of Theorem~\ref{Whitehead}

\begin{proof}[Proof of Theorem~\ref{Whitehead}]
Let $f:S^1\times D^2\rightarrow S^3$ denote a smooth embedding such that $f(S^1\times \{0\})=K$. Set $T=f(S^1\times D^2)$. Up to isotopy, the collection of essential, simple closed, oriented curves in $\partial T$ is parameterized by \[\{m\mu+n\lambda~|~m, n\in \mathbb{Z}~ \textrm{and}~ \textrm{gcd}(m,n)=1\}\] where $\mu$ denotes a meridian in $\partial T$ and $\lambda$ denotes a standard longitude in $\partial T$ coming from a Seifert surface. With this parameterization, the only curves that are null-homologous in $T$ are $\pm\mu$ and the only curves that are null-homologous in $S^3\setminus{\textrm{int}(T)}$ are $\pm\lambda$. Of course $\pm\mu$ will bound embedded disks in $T$, but $\pm\lambda$ will not bound embedded disks in $S^3\setminus{\textrm{int}(T)}$ as $K$ is a non-trivial knot. In other words, the only compressing curves for $\partial T$ in $S^3$ are meridians. 

Suppose now that $C$ is a smooth, simple closed curve in the interior of $T$, and there is a smoothly embedded $2$-disk, say $\Delta$, in $S^3$ such that $\partial \Delta=C$. Since $C$ lies in the interior of $T$, we may assume that $\Delta$ meets $\partial T$ transversely in a finite number of circles. Initially observe that if $\Delta\cap \partial T=\emptyset$, then we can use $\Delta$ to isotope $C$ in the interior of $T$ so that the result of this isotopy is a curve in the interior of $T$ that misses a meridinal disk for $T$.  Now suppose that $\Delta\cap \partial T\neq\emptyset$. We show, in this case too,  $C$ can be isotoped to a curve that misses a meridinal disk for $T$. To this end, let $\sigma$ denote a simple closed curve in $\Delta\cap \partial T$ such that $\sigma$ is innermost in $\Delta$. That is $\sigma$ bounds a sub-disk, $\Delta'$ say, in $\Delta$ and the interior of $\Delta'$ misses $\partial T$. There are two cases, depending on whether or not that $\sigma$ is essential in $\partial T$. If $\sigma$ is essential in $\partial T$, then, as has already been noted, $\sigma$ must be a meridian. As such, $\Delta'$ will be a meridinal disk in $T$ and $C$ misses $\Delta'$. If $\sigma$ is not essential in $\partial T$, then $\sigma$ bounds an embedded $2$-disk, say $D$, in $\partial T$. It is possible that $\Delta$ meets the interior of $D$, but we can still cut and paste $\Delta$ along a sub-disk of $D$ to reduce the number of components in $\Delta\cap \partial T$. Repeating this process yields that if $C$ is smoothly embedded curve in the interior of $T$ and $C$ is unknotted in $S^3$, then $C$ can be isotoped in the interior of $T$ so as to miss a meridinal disk for $T$.

With all this in place, we return to discuss Whitehead double of $K$. Suppose that $F$ is a standard, genus 1 Seifert surface for a double of $K$. See Figure~\ref{doubleSeifert}. The surface $F$ can be viewed as an annulus $A$ with a a $1$-handle attached to it. Here $K$ is a core circle for $A$, and the $1$-handle is attached to $A$ as depicted in Figure~\ref{handle}          

\begin{figure}[h!]
\begin{center}
 \includegraphics[width=8cm]{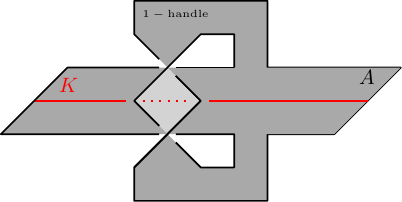}
 \caption{Standard genus 1 Seifert surface $F$ for a double of $K$.}
  \label{handle}
\end{center}
\end{figure}

Observe that $F$ can be constructed so that it lives in the interior of $T$. Now, the curve $C$ that passes once over the $1$-handle  and zero times around $A$ obviously misses a meridinal disk for $T$, and it obviously is unknotted in $S^3$. On the other hand, if $C$ is any other essential simple closed curve in the interior of $F$, then $C$ must go around $A$ some positive number of times. It is not difficult, upon orienting, $C$ can be isotoped so that the strands of $C$ going around $A$ are coherently oriented. As such, $C$ is homologous to some non-zero multiple of $K$ in $T$. This, in turn, implies that $C$ cannot be isotoped in $T$ so as to miss some meridinal disk for $T$. It follows that $C$ cannot be an unknot in $S^3$. 

\end{proof}

%{\bf Final remarks.} We note that the choice of genus one Seifert surface $\Sigma_K$ we made for $K=K_t$ is really not a restriction (see \cite{ET}[Proof of Theorem~$1$]). Moreover, the answer for finding unknotted curves on a fixed Seifert surface of genus one knot $K\subset S^3$ qualitatively doesn't change for the mirror knot $\overline K$. And this way we obtain similar answers for twist knot with odd number of half twists in the twist region by reversing clasp in twist knot in Figure~\ref{TwistKnots} to be right handed and then mirroring the knot. 

\providecommand{\bysame}{\leavevmode\hbox to3em{\hrulefill}\thinspace}
\providecommand{\MR}{\relax\ifhmode\unskip\space\fi MR }
\providecommand{\MRhref}[2]{%
  \href{http://www.ams.org/mathscinet-getitem?mr=#1}{#2}
}
\providecommand{\href}[2]{#2}

\end{document}